\documentclass[final]{siamart1116}

\usepackage{tikz}
\usetikzlibrary{hobby} 
\usetikzlibrary{intersections}
\usetikzlibrary{positioning}

\usepackage{graphicx}
\usepackage{comment}
\usepackage{dsfont}
\usepackage{setspace}
\usepackage[dvips]{epsfig}
\usepackage{subcaption}
\usepackage{amssymb, amsmath, cite}
\usepackage{color}
\usepackage{verbatim}
\usepackage{mathrsfs}
\usepackage{color,hyperref}
\usepackage[top=1in, bottom=0.7in, left=0.7in, right=0.7in]{geometry}
\usepackage{mathtools}
\usepackage{url}
\usepackage{pdfpages}
\usepackage{listings}

\numberwithin{theorem}{section}

\newsiamthm{problem}{Problem}
\newsiamthm{defn}{Definition}
\newsiamthm{assum}{Assumption}
\newsiamthm{remark}{Remark}
\newcommand{\LQCP}{$\text{(LQCP)}_{\mathcal{L}}~$}
\newcommand{\LQCPONE}{$\text{(LQCP)}_{\mathcal{L}_1}~$}

\newcommand{\CC}{{\mathbb C}}

\newcommand{\RR}{{\mathbb R}}

\newcommand{\cC}{{\mathcal C}}

\newcommand{\cL}{{\mathcal L}}

\newcommand{\cN}{{\mathcal N}}
\newcommand{\cNL}{{\mathcal N}_1({\mathcal L}_1)}

\newcommand{\cU}{{\mathcal U}}
\newcommand{\cV}{{\mathcal V}}
\newcommand{\cW}{{\mathcal W}}
\newcommand{\cX}{{\mathcal X}}

\newcommand{\Ks}{K^{\star}}
\newcommand{\Kos}{K^{\star}_1}
\newcommand{\Kts}{K_2^{\star}}
\newcommand{\Kots}{K_{12}^{\star}}

\newcommand{\ol}{\overline}

\newcommand{\p}{\partial}

\renewcommand{\Im}{\textup{Im}}
\renewcommand{\Re}{\textup{Re}}
\newcommand{\Ker}{\textup{Ker}}

\title{The regular indefinite linear quadratic optimal control problem: stabilizable case%
  \thanks{Published electronically Feb. 20 2018.
\funding{Supported by the Natural Sciences and Engineering Research Council of Canada (NSERC).}}}

\author{Marijan Vukosavljev%
  \thanks{Dept. of Electrical and Computer Engineering, University of Toronto, Canada (\email{mario.vukosavljev@mail.utoronto.ca}, \email{broucke@control.utoronto.ca}).}%
  \and
  Angela~P.~Schoellig%
   \thanks{University of Toronto Institute for Aerospace Studies, University of Toronto, Canada (\email{schoellig@utias.utoronto.ca}).}%
  \and
  Mireille~E.~Broucke%
  \footnotemark[2]
}

\headers{Indefinite linear quadratic optimal control}
{Marijan Vukosavljev, Angela P. Schoellig, and Mireille E. Broucke}

\begin{document}
\maketitle

\begin{abstract}
This paper addresses an open problem in the area of linear quadratic optimal control. We consider the regular, infinite-horizon, stability-modulo-a-subspace, indefinite linear quadratic problem under the assumption that the dynamics are stabilizable. Our result generalizes previous works dealing with the same problem in the case of controllable dynamics. We explicitly characterize the unique solution of the algebraic Riccati equation that gives the optimal cost and optimal feedback control, as well as necessary and sufficient conditions for the existence of optimal controls.

\end{abstract}



\section{Introduction}

In this paper we consider the regular, infinite-horizon linear quadratic optimal control problem in which the cost functional is the integral of an {\em indefinite} quadratic form. The {\em regular} linear quadratic (LQ) problem, when the quadratic form in the cost functional is {\em positive definite} in the control variables, has been studied extensively in the literature \cite{BROCKETT70,WONHAM85,MOO89}. It has been especially well studied under the standard assumption, the so-called {\em positive semidefinite} case, when the quadratic form in the cost functional is positive semidefinite in the control and state variables simultaneously. The more general indefinite case imposes no definiteness condition in the control and state variables simultaneously \cite{WIL71, TRE89}. The LQ problem is termed {\em infinite-horizon} if the cost functional is integrated over time from zero to infinity.  Finally, the most typical treatment of the LQ problem is the {\em fixed-endpoint} problem where the state is required to converge to zero as time tends to infinity. The case when no such condition is imposed has also been studied and is referred to as the {\em free-endpoint} problem \cite{TRE89, TRE89b, GEE91}. In fact, an entire family of LQ problems can be obtained by requiring that the state converges to a subspace. This so-called {\em stability-modulo-a-subspace} family of LQ problems includes the fixed- and free-endpoint problems as special cases \cite{TRE89b, GEE91}. For the remainder of the paper, we restrict our attention to the regular and infinite-horizon versions of the problem, for otherwise the optimization problem may yield optimal controllers that are not static linear state feedbacks \cite{WIL86, MOO89}. Also, we focus on stability-modulo-a-subspace, since it is the more general case.
 
Traditionally, a complete solution of any variant of the LQ problem requires to find necessary and sufficient conditions for the existence of a finite optimal cost and optimal controls. Existence of a finite optimal cost is called well-posedness, while existence of an optimal control is called attainability. Further, when they exist, a complete solution  involves determining the optimal cost and an optimal control. Both should be expressed in terms of the given problem data; that is, the system matrices, the instantaneous cost matrices, and the desired subspace. 

In the regular, infinite-horizon, fixed-endpoint, positive semidefinite case, the LQ problem was fully resolved in 1968 by Wonham \cite{WONHAM68, WONHAM85}, resulting in the well known necessary and sufficient conditions involving stabilizability and detectability. The corresponding free-endpoint LQ problem was fully characterized much later \cite{GEE88,TRE02}, resulting in conditions involving output stabilizability, a condition less strict than stabilizability \cite{GEE88,TRE02}. In the regular, infinite-horizon, indefinite case, the fixed-endpoint problem was solved in 1971 by Willems \cite{WIL71}, while the free-endpoint problem and general stability-modulo-a-subspace were addressed in 1989 by Trentelman \cite{TRE89, TRE89b}. Importantly, all of the indefinite cases made use of the assumption that the dynamics are controllable. Moreover the solutions are incomplete in that only sufficient conditions for the existence of a finite optimal cost were given (except for the fixed-endpoint problem). The main contribution of this paper is to extend the above results for the regular, infinite-horizon, stability-modulo-a-subspace, indefinite case of the LQ problem. Rather than assuming controllability, we only require stabilizability.

It is well known that in both the positive semidefinite and indefinite cases of the regular, infinite-horizon,  stability-modulo-a-subspace LQ problem, the optimal cost and optimal controls  are given in terms of a particular solution of the algebraic Riccati equation (ARE) \cite{TRE02, TRE89}. In the treatment of the regular, infinite-horizon, indefinite LQ problem, the controllability assumption is crucial in order 
to utilize the geometry of the set of all real symmetric solutions of the ARE \cite{WIL71,LAN95}. 
In particular, if this solution set is nonempty, there exist a maximal and 
minimal solution of the ARE \cite{LAN95}. The regular, infinite-horizon, fixed-endpoint LQ problem, both definite and indefinite cases, has always been easier in the 
sense that the optimal cost and feedback control law are given in terms of the maximal solution, which is the only 
solution that can stabilize the closed-loop system \cite{WIL71,WONHAM68}. For the regular, infinite-horizon, stability-modulo-a-subspace, indefinite case and
under the assumption of controllability, the optimal cost and feedback control law are given by a real symmetric solution to 
the ARE that depends on both its maximal and minimal solutions \cite{TRE89b}. In contrast, under the stabilizability assumption, 
it is unclear which solution of the ARE to select because the geometry of the set of all real symmetric ARE solutions is less well-behaved. In particular, 
the minimal solution may no longer exist \cite{GOH86,LAN95}. This ambiguity of the correct choice of ARE solution for the regular, infinite-horizon, stability-modulo-a-subspace, indefinite LQ problem under merely stabilizable dynamics was discussed by Geerts 
\cite{GEE89,GEE91}, but it has remained elusive.

In this paper we give the exact form of the optimal feedback that solves the regular, infinite-horizon, stability-modulo-a-subspace, indefinite LQ problem under stabilizable dynamics. Thus we resolve the ambiguity regarding which solution of the ARE to take. Our result requires two assumptions, which are precisely our sufficient conditions for well-posedness: existence of a negative semidefinite solution to the algebraic Riccati
inequality (ARI) and stabilizability of the system dynamics. These assumptions may be compared to the sufficient conditions for well-posedness in \cite{TRE89}:
existence of a negative semidefinite solution to the ARE and controllability of the system dynamics. The first
assumption on existence of a negative semidefinite solution of the ARE or ARI provides for a lower bound on the 
value function, based on a result of Molinari \cite{MOL77}. Our generalization to the ARI is based on an observation by 
Geerts \cite{GEE89}. The generalization to the case when the dynamics are stabilizable proves to be the more difficult challenge, as discussed above. This extension constitutes the central contribution of the paper. Finally, we give necessary and sufficient conditions for optimal controls to exist, which, as pointed out in \cite{TRE89}, 
are nontrivial for regular, infinite-horizon, non-fixed-endpoint, indefinite LQ problems.

As a further validation of the correctness of our results, we recover known results for other variants of the regular, infinite-horizon LQ problem by adding assumptions to match those problems. In the regular, infinite-horizon, stability-modulo-a-subspace, indefinite case, if we assume controllable dynamics, we obtain the same necessary and sufficient conditions for the existence of optimal controls, the same form of the optimal cost, and the same form of the optimal control as stated in \cite{WIL71, TRE89, TRE89b}. In the regular, infinite-horizon, positive semidefinite LQ problem, for both the fixed- and free-endpoint cases, if we assume positive semidefineness, then we again obtain the same necessary and sufficient conditions for the existence of optimal controls, the same form of the optimal cost, and the same form of the optimal control as stated in \cite{TRE02}.

Our resolution of the gap in the LQ literature provides more than just an answer to an academic question. 
Recently, the work in \cite{PACH09} considered a linear term in the state of the cost functional and a 
free-endpoint objective, albeit over the finite-horizon; with a transformation, this cost can be converted 
to an indefinite problem with stabilizable but not controllable dynamics. 
The gap was also recently discussed in \cite{ENG08}, which deals with the cooperative indefinite LQ problem. As such, our result has application to game theoretic formulations and economics. 

The outline of this paper is as follows. In the remainder of this section we will introduce most of the notational conventions 
that will be used. In Section~\ref{sec:prob} we present the problem 
statement. In Section~\ref{sec:ARE} we summarize the key ingredients needed regarding the geometry of the ARE solutions. 
In Section \ref{sec:main} we state and prove our main results. In Section \ref{sec:discuss} we compare our main result
to existing results in the literature. 

{\bf Notation.} \ We use the following notation.
Let $I_n$ be the $n \times n$ identity matrix (the subscript is omitted if the dimension is clear from the context). Let $P^{\dagger}$ denote the (unique) pseudo-inverse of $P \in \RR^{n \times m}$. 
The set of eigenvalues of $A \in \RR^{n \times n}$ is denoted by $\sigma(A)$.
%If $A \in \RR^{n \times n}$ and $\mathcal{V} \subset \RR^n$ is a subspace, 
%$A|_{\mathcal{V}}$ will denote the restriction of $A$ to $\mathcal{V}$.
A subspace $\cV \subset \RR^n$ is $A-$invariant if $A \mathcal{V} \subset \mathcal{V}$.
We use the following subsets of the complex plane: 
$\CC^- := \{s \in \CC \; | \; \Re(s) < 0\}$, $\CC^0 := \{s \in \CC \; | \; \Re(s) = 0\}$, and
$\CC^+ := \{s \in \CC \; | \; \Re(s) > 0\}$.
Given a real monic polynomial $p$ there is a unique factorization $p = p_{-} \cdot p_0 \cdot p_+ $ into real monic polynomials with $p_-$, $p_0$, and $p_+$ having all roots in $\CC^-$, $\CC^0$, and $\CC^+$, respectively. Then if $A \in \RR^{n \times n}$ and if $p$ is its characteristic polynomial, then we define the spectral subspaces
$\mathcal{X}^-(A) := \Ker (p_-(A))$, $\mathcal{X}^0(A) := \Ker (p_0(A))$, and
$\mathcal{X}^+(A) := \Ker (p_+(A))$. Each of these subspaces are $A-$invariant and the restriction of $A$ to $\mathcal{X}^-(A)(\mathcal{X}^0(A),\mathcal{X}^+(A))$ has characteristic polynomial $p_-(p_0,p_+)$.
For two subspaces $\cV$ and $\cW$, let $\cV \oplus \cW$ denote their direct sum and let $\cV \sim \cW$ denote that they are isomorphic.
%A subset $\CC_g \subset \CC$ is called symmetric if $a+bi \in \CC_g \Leftrightarrow a-bi \in \CC_g$. If $\CC_g$ is given then we define $\CC_b := \CC \backslash \CC_g$. If $A$ is a real $n\times n$ matrix and if $p$ is its characteristic polynomial then, again, $p$ can be factored uniquely into $p = p_g \cdot p_b$, where $p_g$ and $p_g$ are real monic polynomials with all roots in $\CC_g$ and $\CC_g$ respectively. We denote $\mathcal{X}_g := \Ker(p_g(A))$ and $\mathcal{X}_b := \Ker(p_b(A))$. Again these subspaces are $A-$invariant and the restriction of $A$ to $X_g(A)$($X_b(A)$) has characteristic polynomial in $p_g$($p_b$).
For an arbitrary matrix $A \in \RR^{n \times n}$ and subspace $\cV \subset \RR^n$ we define the subspace $\langle A \; | \; \cV \rangle := \cV + A \cV + \ldots A^{n-1} \cV$, and by further writing $\cV = \Ker(W)$ for some $W \in \RR^{p \times n}$ we also define $\langle \cV \; | \; A \rangle := \Ker (W) \cap \Ker (WA) \ldots \cap \Ker (WA^{n-1})$.
For a linear time-invariant system, $\dot{x} = Ax + Bu$, the controllable subspace will be denoted in the usual way $\langle A \; | \; \Im (B) \rangle$.
%$ := \Im (B) + A \Im (B) + \ldots + A^{n-1} \Im (B)$.
If there is an output $y = Cx$, then $\langle \Ker (C) \; | \; A \rangle$
% := \Ker (C) \cap \Ker (CA) \ldots \cap \Ker (CA^{n-1})$
denotes the unobservable subspace of $(C,A)$.
%Furthermore, given a symmetric subset $\CC_g \subset \CC$ we denote $\mathcal{X}_{det} := \langle \Ker(C) \; | \; A \rangle \cap \mathcal{X}_b(A)$ the undetectable subspace of $(C,A)$ with respect to $\CC_g$. The pair $(C,A)$ is called detectable with respect to $\CC_g$ if $\langle \Ker(C) \; | \; A \rangle \subset \mathcal{X}_g(A)$.
If $M$ is a real $n \times n$ matrix and $\cV$ is a subspace of $\RR^n$, then $M^{-1}\cV := \{ x \in \RR^n \; | \; Mx \in \cV \}$. If $\cV$ is a subspace of $\RR^n$ then $\cV^{\perp}$ denotes its orthogonal complement with respect to the standard Euclidean inner product.

Let $\RR^+ := \{t \in \RR \; | \; t \geq 0\}$ and $\RR^e := \RR \cup \{-\infty, +\infty\}$. Additionally, given a function $f: \RR \rightarrow \RR$, the statement that $\lim_{t \rightarrow \infty} f(t)$ exists in $\RR^e$ means that $\lim_{t \rightarrow \infty} f(t)$ is either equal to a real number, $\infty$, or $-\infty$ in the usual sense.

We denote the space of all measurable vector-valued functions on $\RR^+$ that are locally square integrable as
$L_{2,loc}^{m}(\RR^+) = \left\{\left. u: \RR^+ \rightarrow \RR^m \; \right| \; (\forall T \geq 0) 
\int_0^T u(t)^{\top}u(t) \, dt < \infty \right\}$.
Let $d_{\cL} : \RR^n \rightarrow [0,\infty)$ denote the function giving the minimum Euclidean distance from 
a point to a set $\mathcal{L} \subset \RR^n$. 

Given a quadratic form on $\RR^n$, $\omega : \RR^n \rightarrow \RR$, it is said to be {\em positive definite} if for all $x \in \RR^n$, $\omega(x) \geq 0$, and $\omega(x) = 0$ if and only if $x = 0$; {\em positive semidefinite} if for all $x \in \RR^n$, $\omega(x) \geq 0$; {\em negative definite} if $-\omega$ is positive definite; {\em negative semidefinite} if $-\omega$ is positive semidefinite; and {\em indefinite} if $\omega$ is neither positive semidefinite nor negative semidefinite. Writing $\omega(x):= x^{\top}P x$ for some symmetric matrix $P \in \RR^{n \times n}$, we say that the matrix $P$ is positive definite if the quadratic form $\omega$ is positive definite and so on. We write $P > 0$, $P \geq 0$, $P < 0$, and $P \leq 0$ if 
the matrix is positive definite, positive semidefinite, negative definite, and negative semidefinite, 
respectively. 
Given symmetric matrices $P, Q \in \RR^{n \times n}$, we write $P < Q$ if $Q - P > 0$, and likewise for 
the other inequalities. Let $\Lambda$ denote a subset of the set of all symmetric matrices in $\RR^{n \times n}$. 
We say that \textit{$M^+$ ($M^-$) is the maximal (minimal) element on $\Lambda$} if $M^+ \in \Lambda$ ($M^- \in \Lambda$) 
and for all $M \in \Lambda$, $M \leq M^+$ ($M \geq M^-$). The maximal and minimal elements, which are called the \textit{extremal elements on $\Lambda$}, are unique if they exist since $\Lambda$ forms a partially ordered set.

\section{Problem Statement}
\label{sec:prob}

We consider the linear control system
\begin{equation} 
\label{eq:sysgeneral}
\dot{x} = Ax + Bu, \qquad x(0) = x_0,
\end{equation}
where $x\in \RR^n$ and $u \in \RR^m$. For a control function $u \in L_{2,loc}^m(\RR^+)$, let $x(\cdot ; x_0,u)$ 
denote the state trajectory of \eqref{eq:sysgeneral} starting at $x_0 \in \RR^n$. Then for $T \geq 0$, the 
{\em cost function} is
\begin{equation}  
\label{eq:finTimeCostFunctGeneral}
J_T(x_0,u) = \int_0^T \omega(x(t;x_0,u)) \, dt
\end{equation}
with a quadratic {\em instantaneous cost} 
\begin{equation} 
\label{eq:qformMostGeneral}
\omega(x,u) := 
x^{\top}Qx + u^{\top}R u = 
\begin{bmatrix}x^{\top} & u^{\top} \end{bmatrix} W 
\begin{bmatrix} x \\ u \end{bmatrix} \,, \qquad 
W := \begin{bmatrix} Q & 0 \\ 0 & R \end{bmatrix} \,, \qquad R = I_m.
\end{equation}
We allow $Q$ to be indefinite, whereas $R := I_m > 0$. More general quadratic cost functions can be considered, 
but they can be converted via a feedback transformation to the form we use here, as in Chapter 10 of \cite{TRE02}. 
This feedback transformation does not affect solvability of the problem; hence, there is no loss of 
generality in our choice of $W$.

Because $W$ may be indefinite, we define the set of control inputs that yield a cost that is either finite, $\infty$, or $-\infty$:
\begin{equation}
\cU(x_0) := 
\left\{ \left. u \in L_{2,loc}^m(\RR^+) \; \right| \; \lim_{T \rightarrow \infty} J_T(x_0,u) 
\text{ exists in } \RR^e\right\}. 
\end{equation}
\noindent Let $\mathcal{L} \subset \RR^n$ be a subspace. The set of permissible control inputs such that the state asymptotically converges to $\cL$ is
\begin{equation} 
\label{eq:controlsetAnyEPGeneral}
\cU_{\cL}(x_0) := 
\left\{ \left. u \in \cU(x_0) \; \right| \; \lim_{t \rightarrow \infty} d_{\cL}( x(t;x_0,u) ) = 0 \right\} \,.
\end{equation}
For $u \in \cU_{\cL}(x_0)$, we define
\begin{equation}
J(x_0, u) := \lim_{T \rightarrow \infty} J_T(x_0, u).
\end{equation}
We define the {\em optimal cost} or {\em value function} to be 
\begin{equation} 
\label{eq:valuefuncAnyEPGeneral}
V_{\mathcal{L}}(x_0) := \inf \{ J(x_0,u) \; | \; u \in \cU_{\cL}(x_0) \} \,.
\end{equation}

Now we define the linear quadratic optimal control problem with stability-modulo-$\cL$ \LQCP.
\begin{problem}[\LQCP] 
Consider the system \eqref{eq:sysgeneral} with the quadratic cost criterion \eqref{eq:finTimeCostFunctGeneral}. 
Let $\mathcal{L} \subset \RR^n$ be a given subspace. For all $x_0 \in \RR^n$, find the optimal cost 
$V_{\cL}(x_0)$ and an optimal control $u^{\star} \in \cU_{\cL}(x_0)$ such that 
$V_{\cL}(x_0) = J(x_0,u^{\star})$. 
\end{problem}

The \LQCP is called {\em regular} (as opposed to singular) if $R > 0$. It is called 
{\em positive semidefinite} if $\omega$ is positive semidefinite on $\RR^{n + m}$, and {\em indefinite} otherwise. 
If $\mathcal{L}= \RR^n$, the \LQCP is called a {\em free-endpoint problem}, and if 
$\mathcal{L} = 0$, it is called a {\em fixed-endpoint problem}. 
We are particularly interested in characterizing two properties of the \LQCP.
\begin{defn}
We say the \LQCP is {\em well-posed} if for all $x_0 \in \RR^n$, $V_{\mathcal{L}}(x_0) \in \RR$.
We say the \LQCP is {\em attainable} if for all $x_0 \in \RR^n$, 
there exists a control $u^{\star} \in \cU_{\cL}(x_0)$ such that $V_{\mathcal{L}}(x_0) = J(x_0, u^{\star})$.
Such an input is called {\em optimal}. 
We say the \LQCP is {\em solvable} if it is both well-posed and attainable. 
\end{defn}

\section{Preliminaries}
\label{sec:ARE}

The main results on the \LQCP are centered on the algebraic Riccati equation (ARE): 
\begin{equation} 
\label{eq:ARE}
\phi(K) := A^{\top}K + KA + Q - KBB^{\top}K = 0 \,.
\end{equation}
The algebraic Riccati inequality (ARI) is given by $\phi(K) \ge 0$. For convenience, 
we define 
\begin{equation}A(K) := A - B B^T K.
\end{equation}
\noindent Also we define the following solution sets: 
\begin{align*}
\Gamma &:= \{K \in \RR^{n \times n} ~|~ K = K^{\top}, \phi(K) \geq 0\}, \\ 
\p \Gamma &:= \{ K \in \RR^{n\times n}  ~|~ K = K^{\top} , \; \phi(K) = 0 \}, \\
\Gamma_- &:= \{ K \in \Gamma \; | \; K \leq 0 \}.
\end{align*} 
The geometry of the solutions to the ARE can be studied in both the controllable and stabilizable cases; 
see, in particular, Chapters 7 and 8 of \cite{LAN95} and also \cite{TRE89}. First we consider the case when 
$(A,B)$ is controllable. The next result summarizes what is known about the extremal solutions in 
$\Gamma$ and in $\p \Gamma$. 

\begin{theorem}
%[\cite{LAN95}, Theorem 14(b) \cite{SCH91}] 
\label{thm:extremal}
Suppose $(A,B)$ is controllable. 
\begin{itemize}
\item[(i)]
If $\Gamma \neq \emptyset$, then the maximal and minimal solutions in $\Gamma$ exist, 
$\p \Gamma \neq \emptyset$, its maximal and minimal solutions exist, 
and they are identical to the maximal and minimal solutions in $\Gamma$. 
\item[(ii)]
If $\p \Gamma \neq \emptyset$, then its maximal and minimal solutions $K^+, K^- \in \p \Gamma$ 
satisfy: $\forall K \in \p \Gamma$, $K^- \leq K \leq K^+$. Moreover, they are the unique solutions 
in $\p \Gamma$ such that $\sigma(A(K^+)) \subset \CC^- \cup \CC^0$ and $\sigma(A(K^-)) \subset \CC^+ \cup \CC^0$.  
\end{itemize}
\end{theorem}

\begin{proof}
The first statement is Theorem 14(b) in \cite{SCH91}. 
The second statement was proved in \cite{WIL71}. See also Theorem 7.5.1, p. 168, in \cite{LAN95}. 
\end{proof}

If $\p \Gamma \neq \emptyset$, define the {\em gap} of the ARE to be $\Delta:= K^+ - K^-$. 
Let $\Omega$ denote the set of all $A(K^-)-$invariant subspaces contained in $\mathcal{X}^+(A(K^-))$. 
The following theorem was first proven by Willems \cite{WIL71}; see also \cite{LAN95}.

\begin{theorem}[Theorem 3.1, \cite{TRE89}] 
\label{thm:geoAREproj}
Let $(A,B)$ be controllable and suppose $\partial\Gamma \neq \emptyset$. 
If $\mathcal{V} \subset \Omega$, then $\RR^n = \mathcal{V} \oplus \Delta^{-1}(\mathcal{V} ^{\perp})$. 
There exists a bijection $\gamma: \Omega \rightarrow \partial\Gamma$ defined by
\begin{equation}
\gamma(\mathcal{V}) := K^- P_{\mathcal{V}} + K^+ (I_n - P_{\mathcal{V}}),
\end{equation}
where $P_{\mathcal{V}}$ is the projection onto $\mathcal{V}$ along 
$ \Delta^{-1}(\mathcal{V} ^{\perp})$. If $K = \gamma(\mathcal{V})$, then
$\mathcal{X}^+(A(K)) = \mathcal{V}$, $\mathcal{X}^0(A(K)) = \Ker(\Delta)$, and
$\mathcal{X}^-(A(K)) = \mathcal{X}^-(A(K^+))\cap \Delta^{-1}(\mathcal{V} ^{\perp})$.
\end{theorem}

An application of Theorem~\ref{thm:geoAREproj} is the main result of \cite{TRE89b}, which provides a solution 
of the \LQCP when $(A,B)$ is controllable. To state the sufficient condition for well-posedness, an additional definition is needed from \cite{TRE89b}: for a given subspace $\cL \subset \RR^n$ and symmetric matrix $K \in \RR^{n \times n}$, $K$ is said to be {\em negative semidefinite on $\cL$} if for all $x \in \cL$, $x^{\top}K x \leq 0$, and $ x^{\top}K x = 0$ if and only if $Kx = 0$. Notice that $K \leq 0$ implies that for all $\cL \subset \RR^n$, $K$ is negative semidefinite on $\cL$. To see this, fix $\cL \subset \RR^n$ and note that $K \leq 0$ implies that there exists $H \in \RR^{p \times n}$ for some $p$ such that $K = -H^{\top}H$. Then for all $x \in \cL \subset \RR^n$, obviously $x^{\top}K x \leq 0$, $Kx = 0$ implies $x^{\top}Kx = 0$, and
\begin{equation} \label{eq:negsemidefquadform}
x^{\top} K x = - (Hx)^{\top}(Hx) = 0 \;\; \Leftrightarrow \;\;
Hx = 0 \;\; \Rightarrow \;\; -H^{\top}(Hx) = Kx = 0.
\end{equation}

\begin{theorem}[Theorem 4.1, \cite{TRE89b}]
\label{thm:LQCPcontrollable}
Let $(A,B)$ be controllable. Assume $\p \Gamma \neq \emptyset$ and $K^-$ is negative semidefinite on $\cL$.
Then we have
\begin{enumerate}
\item[(i)]
For all $x_0 \in \RR^n$, $V_{\cL}(x_0)$ is finite.
\item[(ii)]
For all $x_0 \in \RR^n$, $V_{\cL}(x_0) = x_0^{\top} \Ks x_0$, where $\Ks := \gamma(\cN(\cL))$ and
$\cN(\cL) := \langle \cL \cap \Ker( K^-) \; | \; A(K^-) \rangle \cap \cX^+(A(K^-))$.
\item[(iii)]
For all $x_0 \in \RR^n$, there exists an optimal input $u^{\star}$ if and only if $\Ker(\Delta) \subset \cL \cap \Ker(K^-)$.
\item[(iv)]
If $\Ker(\Delta) \subset \cL \cap \Ker(K^-)$, then for each $x_0 \in \RR^n$, there exists exactly one optimal
input $u^{\star}$, and it is given by the feedback $u^{\star} = -B^{\top} \Ks x$.

\end{enumerate}
\end{theorem}

This paper can be regarded as a generalization of the previous result to the stabilizable case. That is, we require weaker assumptions for the sufficient condition of well-posedness to be able to provide the form of the value function, necessary and sufficient conditions for attainability, and the form of the optimal control. Our new assumptions involve the stabilizability of $(A,B)$ rather than controllability, and the existence of a negative semidefinite solution to the ARI rather than imposing that specifically $K^-$, a solution to the ARE, is negative semidefinite on $\cL$. Because necessary and sufficient conditions for well-posedness are still an open problem, note that we have not attempted to generalize our second condition in terms of the existence of an ARI solution that is negative semidefinite on $\cL$. Regardless, the main technical obstacle
is that there is no direct generalization of Theorem~\ref{thm:geoAREproj} to the stabilizable case; indeed the minimal 
solution $K^-$ may not exist in this case.

Now supposing that $(A,B)$ is stabilizable, we can write the system \eqref{eq:sysgeneral} in the Kalman controllability decomposition. Let $\mathcal{C} = \langle A \; | \; \Im(B) \rangle \subset \RR^n$ 
be the controllable subspace with dimension $n_1 \leq n$. Also, let $\cX_2$ be any complement such that 
\begin{equation} \label{eq:KalmanDecomp}
\RR^n = \cC \oplus \cX_2. 
\end{equation}
\noindent Then the system matrices have the block form: 
\begin{equation} 
\label{eq:sysctrdecomp}
A = \begin{bmatrix} A_1 & A_{12} \\ 0 & A_2 \end{bmatrix}, \qquad 
B = \begin{bmatrix} B_1 \\ 0 \end{bmatrix}.
\end{equation}
It can be shown that coordinate transformations only affect the solutions $K \in \p \Gamma$ of the \LQCP (in any endpoint case) up to a 
congruent transformation, so there is no loss of generality to assume that $(A,B)$ already has the form \eqref{eq:sysctrdecomp}. 
If we write the symmetric matrices $Q$ and $K$ in block form
\begin{equation}
\label{eq:QKblock}
Q       = \begin{bmatrix} Q_1 & Q_{12} \\ Q_{12}^{\top} & Q_2 \end{bmatrix} \,, \quad  
K       = \begin{bmatrix}K_1 & K_{12} \\ K_{12}^{\top} & K_2 \end{bmatrix} \,, 
\end{equation}
then $\phi(K)$ also can be decomposed in block form:
\begin{equation}
\label{eq:phiblock}
\phi(K) = \begin{bmatrix} \phi_1(K_1) & A_1(K_1)^{\top}K_{12} + K_{12}A_2 + K_1A_{12} + Q_{12} \\ 
 * & A_2^{\top}K_2 + K_2A_2 + K_{12}^{\top}A_{12} + A_{12}^{\top}K_{12} + Q_2 - K_{12}^{\top}B_1B_1^{\top}K_{12} 
          \end{bmatrix} \,.
\end{equation}
\noindent We note that $\phi(K)$ is symmetric, and $\phi_1(K_1)$ is defined below in \eqref{eq:ARE_1}. Let 
\begin{equation}
A_1(K_1) := A_1 - B_1 B_1^{\top} K_1.
\end{equation} 
\noindent Then $\phi(K) = 0$ gives rise to three equations 
\begin{align}
\phi_1(K_1) := A_1^T K_1 + K_1 A_1 + Q_1 - K_1 B_1 B_1^{\top} K_1 &= 0, 
\label{eq:ARE_1} \\
A_1(K_1)^{\top}K_{12} + K_{12}A_{2} &= -(Q_{12} + K_1A_{12}), 
\label{eq:ARE_12} \\
A_2^{\top}K_2 + K_2A_2 &= K_{12}^{\top}B_1B_1^{\top}K_{12} - K_{12}^{\top}A_{12} - A_{12}^{\top}K_{12} - Q_2 \,.
\label{eq:ARE_2}
\end{align}
The first equation \eqref{eq:ARE_1} is a quadratic equation with $(A_1,B_1)$ controllable. 
Its solutions $K_1$ are decoupled from $K_{12}$ and $K_2$, so this lower order ($n_1 \times n_1$) ARE equation can be solved first. The relevant solution sets are denoted as:
\begin{align*}
\Gamma_1 &:= \{K_1 \in \RR^{n_1 \times n_1} ~|~  K_1^{\top} = K_1, \; \phi_1(K_1) \ge 0 \},\\ 
\p \Gamma_1  &:= \{K_1 \in \RR^{n_1 \times n_1}  ~|~  K_1^{\top} = K_1, \; \phi_1(K_1) = 0 \}, \\
\Gamma_{1-} &:= \{K_1 \in \Gamma_1 ~|~  K_1 \le 0 \},\\ 
\p \Gamma_{1-} &:= \{K_1 \in \p \Gamma_1 ~|~  K_1 \le 0 \}. 
\end{align*}
\noindent Using any solution $K_1 \in \p \Gamma_1$, if it exists, \eqref{eq:ARE_12} is a linear (Sylvester) equation for $K_{12}$ which may 
have no solutions, infinitely many solutions, or a unique solution. 
The third equation \eqref{eq:ARE_2} is also a linear (Sylvester) equation. Using any solution $K_{12}$, 
if it exists, gives a unique solution to $K_2$. To see this, recall that if $M_1 \in \RR^{n_1 \times n_1}$, $M_2 \in \RR^{n_2 \times n_2}$, and $M_3 \in \RR^{n_1 \times n_2}$ are given matrices, then the Sylvester equation $M_1 X + X M_2 = M_3$ has a unique solution $X \in \RR^{n_1 \times n_2}$ exactly when $\sigma(M_1) \cap \sigma(-M_2) = \emptyset$ \cite{GANTMACHER}.  Because stabilizability of $(A, B)$ implies $\sigma(A_2) \subset \CC^-$, then by applying the Sylvester solvability criteria to \eqref{eq:ARE_2}, we have that $\sigma(A_2^{\top}) \cap \sigma(-A_2) = \emptyset$, and so $K_2$ is unique for any given $K_{12}$.

In preparation for characterizing the existence and form of the value function analogously to Theorem \ref{thm:LQCPcontrollable} (i) and (ii), we consider existence of extremal solutions in $\p \Gamma$. It is known that when 
$(A,B)$ is stabilizable, then the maximal solution $K^+ \in \p \Gamma$ exists,
whereas the minimal solution $K^-$ may not exist. 

\begin{theorem}[Theorem 2.1, \cite{GOH86}; Theorem 7.9.3, p. 195, \cite{LAN95}]
\label{thm:gohberg}
Suppose $(A,B)$ is stabilizable and $\p \Gamma \neq \emptyset$. Then the unique maximal solution $K^+ \in \p \Gamma$ 
exists. Moreover, $\sigma(A(K^+)) \subset \CC^- \cup \CC^0$. 
\end{theorem}

To obtain a generalization of Theorem~\ref{thm:LQCPcontrollable} to the stabilizable case, one of the major steps in the 
sequel is to apply Theorem~\ref{thm:LQCPcontrollable} to the controllable subsystem $(A_1,B_1)$ and its ARE \eqref{eq:ARE_1}. 
Theorem~\ref{thm:LQCPcontrollable} requires that the minimal solution $K_1^-$ of \eqref{eq:ARE_1} exists and is negative 
semidefinite on $\cL$ within the controllable subspace. The following lemma provides for the existence
of this minimal, negative semidefinite solution. 

\begin{lemma}
\label{lem:wellposedgivesnegsdefK1}
Suppose $(A,B)$ is stabilizable, $\Gamma_- \neq \emptyset$, and the state space is decomposed as in \eqref{eq:KalmanDecomp}. Then the minimal solution 
$K_1^- \in \p \Gamma_{1-}$ exists. 
\end{lemma}

\begin{proof}
Let $K \in \Gamma_-$ so that $\phi(K) \geq 0$ and $K \leq 0$.  Consider $K$, $Q$, and $\phi(K)$ in block form \eqref{eq:QKblock}-\eqref{eq:phiblock}. Applying Theorem~\ref{thm:schur} to both $K$ and $\phi(K)$, 
we obtain $\phi_1(K_1) \geq 0$ and $K_1 \leq 0$, which implies $K_1 \in \Gamma_{1-} \neq \emptyset$. Since also 
$(A_1,B_1)$ is controllable, we can apply Theorem~\ref{thm:extremal}(i) to conclude $K_1^+, K_1^- \in \Gamma_1$, 
the maximal and minimal solutions, exist. Moreover $\p \Gamma_1 \neq \emptyset$ and its maximal and minimal 
elements are precisely $K_1^+$ and $K_1^-$. Because $K_1 \leq 0$, $K_1^- \in \Gamma_1$ is minimal, and $K_1, K^-_1 \in \Gamma_1$, 
we have that $K_1^- \leq K_1 \leq 0$. That is, $K^-_1 \in \p \Gamma_{1-}$, as desired. 
\end{proof} 

\section{Solution of the \LQCP}
\label{sec:main}

In this section we present the solution of the \LQCP. That is, we give sufficient conditions for well-posedness, the form of the value function, necessary and sufficient conditions for attainability, and form of the optimal control. We assume that $\cL \subset \RR^n$ is a given subspace. Well-posedness and the form of the value function are addressed through the following sufficient condition, which are also found in \cite{GEE89, GEE91}.

\begin{assum} 
\label{assum:wellposed}
We assume that $(A,B)$ is stabilizable and $\Gamma_- \neq \emptyset$.
\end{assum}

The following theorem states that the value function is given in terms of a quadratic form of a particular solution to the ARE.
\begin{theorem}
[Theorem 2.1 \cite{GEE89}, Lemma 5 \cite{MOL77}] 
\label{thm:valuefunction}
Consider the \LQCP and suppose Assumption~\ref{assum:wellposed} holds. 
Then there exists a unique $\Ks \in \p \Gamma$ such that for all $x_0 \in \RR^n$, 
$V_{\cL}(x_0) = x_0^{\top} \Ks x_0$. 
\end{theorem}

Next we turn to the form of $\Ks$. Our approach is to choose a suitable basis based on the Kalman controllability 
decomposition \eqref{eq:KalmanDecomp} and on Theorem~\ref{thm:geoAREproj}, following the same method in \cite{TRE89}. 
Then we systematically determine each of the blocks of $\Ks$. First we determine $\Kos$ using results from \cite{TRE89b}; 
second, we compute $\Kots$ assuming $\Kos$ is known; finally, we compute $\Kts$ assuming $\Kots$ is known. 
Now we give a more detailed roadmap on how the technical results are obtained.

The choice of $\Kos$ is resolved by applying Theorem~\ref{thm:LQCPcontrollable} to the controllable subsystem. 
We construct a smaller optimal control problem on the controllable subsystem. Intuitively, the smaller 
optimal control problem should be equivalent to the original \LQCP for initial conditions in the controllable subspace. 
After proving this equivalence, we apply Theorem~\ref{thm:LQCPcontrollable} to obtain $\Kos = \ol{K}_1$, where
$\ol{K}_1$ is defined in \eqref{eq:Kbar1} below.
Next, we fix the choice of $\Kos$ that solves \eqref{eq:ARE_1} and turn to the solution set of \eqref{eq:ARE_12}. 
Generally, this linear Sylvester equation may have an infinite number of solutions, making the choice of $\Kots$ 
nontrivial to determine. However, once $\Kots$ is determined, then $\Kts$ is uniquely determined from the linear 
Sylvester equation \eqref{eq:ARE_2}, since $(A,B)$ is stabilizable. Thus $\Kots$ is the main obstacle. Interestingly, 
under a restrictive regularity assumption introduced in \cite{GOH86}, the solution set of \eqref{eq:ARE_12} collapses to 
a single element. On the other hand, Theorem~\ref{thm:valuefunction} states that $\Kots$ exists without the regularity 
assumption. We forego the assumption and search for a more general principle that can resolve the choice of $\Kots$. 

Our approach involves exploiting the structure within the Kalman controllability decomposition, similarly as in \cite{TRE89}. Based on a 
modal decomposition of $A_1(\ol{K}_1)$, the Sylvester equation \eqref{eq:ARE_12} with 
$K_1 = \Kos$ splits into three decoupled linear Sylvester equations \eqref{eq:ARE_12_1}-\eqref{eq:ARE_12_3}. The 
problematic part of $\Kots$, denoted $K_{12,1}^*$ is then isolated to \eqref{eq:ARE_12_1} only. Regarding the solution of 
\eqref{eq:ARE_12_1}, it is well known (see Theorem 10.13 of \cite{TRE02}) that for stabilizable systems with positive 
semidefinite cost in the free endpoint case, the solution of the ARE is given by the smallest positive 
semidefinite solution in $\p \Gamma$. Also, $0 \in \Gamma$ if and only if $Q \geq 0$ 
(see  for example equation (1.16) of \cite{GEE91}) and so $0 \in \Gamma_-$ and $x_0^{\top}0 x_0 = 0$ gives a lower 
bound on the value function. Using the previous two observations, we find through repeated trials that 
$\Kots = 0$ in the positive semidefinite case. At this point we make a guess that the same form of $\Kots$ 
would arise in the indefinite case. Finally, we unambiguously deduce that $\Kots = 0$. 

Once we have fully characterized the form of $\Ks$, obtaining necessary and sufficient conditions for attainability 
follows analogously to the proof presented in \cite{TRE89,TRE89b}. We require only a few augmentations to account 
for the uncontrollable (but stable) dynamics. Now we proceed to the actual development.

The first step is to fix a suitable basis so that the blocks of $\Ks$ can be computed. 
Consider the Kalman controllability decomposition \eqref{eq:KalmanDecomp}, and suppose Assumption~\ref{assum:wellposed} holds. 
Then by Lemma~\ref{lem:wellposedgivesnegsdefK1},
the unique minimal solution $K^-_1 \in \p \Gamma_1 \neq \emptyset$ exists and $K^-_1 \le 0$. 
Similarly, because $(A_1,B_1)$ is controllable and $\p \Gamma_1 \neq \emptyset$, we can apply 
Theorem~\ref{thm:extremal} to obtain the unique maximal solution $K^+_1 \in \p \Gamma_1$. 
Let $\Delta_1 := K_1^+ - K_1^-$ be the gap associated with \eqref{eq:ARE_1}, the ARE in the controllable subspace. 
Following \cite{TRE89, TRE89b}, we can further decompose the 
controllable subspace based on Theorem~\ref{thm:geoAREproj}. To that end, define the following subspaces of $\RR^{n_1}$:
\begin{eqnarray}
\mathcal{L}_1 & := & \cL \cap \cC \label{eq:L1} \\ 
\cN_1(\cL_1)  & := & \langle \mathcal{L}_1 \cap \Ker( K_1^-) \; | \; A_1(K_1^-) \rangle 
                                           \cap \mathcal{X}^+(A_1(K_1^-)) \label{eq:N1} \,.
\end{eqnarray}
\noindent Here and for the remainder of this section, for simplicity we do not notationally differentiate a subspace that can belong to various vector spaces of different dimensions. For example, although technically $\cL \cap \cC \subset \RR^n$, we can view $\cL_1$ as a subspace of $\RR^{n_1} \sim \cC$.

Let $P_{\cNL} : \RR^{n_1} \rightarrow \cNL$ 
be the projection onto $\cNL$ along $\Delta_1^{-1} (\cN_1(\cL_1)^{\perp})$. 
Because $\cN_1(\cL_1)$ is an $A_1(K_1^-)$-invariant subspace contained in $\cX^+(A_1(K_1^-))$
for any $\cL_1$, we can apply Theorem~\ref{thm:geoAREproj} to obtain a solution $\ol{K}_1 \in \p \Gamma_1$ 
of the ARE of the form
\begin{equation}
\label{eq:Kbar1}
\ol{K}_1 := \gamma(\cNL) = K_1^- P_{\cNL} + K_1^+ (I_{n_1} - P_{\cNL}).
\end{equation}
\begin{comment}
It can be verified from the definition of $A_1(K_1^-)$ (see also the appendix) that 
\begin{equation} 
\label{eq:subsN2}
\mathcal{N}_1(\mathcal{L}_1) = \langle \mathcal{L}_1 \cap \Ker (K_1^-) \; | \; A_1 \rangle \cap \mathcal{X}^+(A_1).
\end{equation}
\end{comment}

Following Theorem~\ref{thm:geoAREproj}, define the following subspaces in 
$\cC \sim \RR^{n_1}$:
\begin{align}
\mathcal{X}_{1,1} &:= 
\mathcal{X}^+(A_1(\ol{K}_1))     = \cN_1(\cL_1), \\
\mathcal{X}_{1,2} &:= 
\mathcal{X}^0(A_1(\ol{K}_1)) = \Ker (\Delta_1)  , \\
\mathcal{X}_{1,3} &:= 
\mathcal{X}^-(A_1(\ol{K}_1)) = \mathcal{X}^-(A_1(K_1^+)) \cap \Delta_1^{-1} (\cN_1(\cL_1)^{\perp}) .
\end{align}
%Also let $\cX_2$ be any complement such that $\RR^n = \cC \oplus \cX_2$.
Then the state space decomposition \eqref{eq:KalmanDecomp} splits further into
\begin{equation}
\label{eq:Rnsplit}
\RR^n = \cX_{1,1} \oplus \cX_{1,2} \oplus \cX_{1,3} \oplus \cX_2 \,.
\end{equation}
Let $n_{1,i} := \text{dim}(\cX_{1,i})$ for $i = 1,2,3$ so that $n_1 = n_{1,1} + n_{1,2} + n_{1,3} \leq n$.
Without loss of generality (after a change of coordinates), the system matrices have the block form
\begin{equation}
\label{eq:ABblock4}
A = \begin{bmatrix} A_1      & A_{12} \\ 0 & A_2 \end{bmatrix}
  = \begin{bmatrix} A_{1,11} & A_{1,12} & A_{1,13} & A_{12,1} \\
                    A_{1,21} & A_{1,22} & A_{1,23} & A_{12,2} \\
                    A_{1,31} & A_{1,32} & A_{1,33} & A_{12,3} \\
                           0 & 0        & 0 & A_2 \end{bmatrix}, \; \; 
B = \begin{bmatrix} B_1 \\ 0 \end{bmatrix}
   = \begin{bmatrix} B_{1,1} \\ B_{1,2} \\ B_{1,3} \\ 0 \end{bmatrix}.
\end{equation}
The cost matrix $Q$ and each $K \in \Gamma$ 
have the block form
\begin{equation} 
\label{eq:QKblock4}
Q = \begin{bmatrix}Q_1 & Q_{12} \\ Q_{12}^{\top} & Q_2 \end{bmatrix}
  = \begin{bmatrix} Q_{1,11} & Q_{1,12} & Q_{1,13} & Q_{12,1} \\
                               Q_{1,12}^{\top} & Q_{1,22} & Q_{1,23} & Q_{12,2} \\
                               Q_{1,13}^{\top} & Q_{1,23}^{\top} & Q_{1,33} & Q_{12,3} \\
                               Q_{12,1}^{\top} & Q_{12,2}^{\top} & Q_{12,3}^{\top} & Q_2 \end{bmatrix}, \; \;
K = \begin{bmatrix} K_1 & K_{12} \\ K_{12}^{\top} & K_2 \end{bmatrix}
  = \begin{bmatrix} K_{1,11} & K_{1,12} & K_{1,13} & K_{12,1} \\
                               K_{1,12}^{\top} & K_{1,22} & K_{1,23} & K_{12,2} \\
                               K_{1,13}^{\top} & K_{1,23}^{\top} & K_{1,33} & K_{12,3} \\
                               K_{12,1}^{\top} & K_{12,2}^{\top} & K_{12,3}^{\top} & K_2 \end{bmatrix}.
\end{equation}
Our goal is to compute all of the blocks in \eqref{eq:QKblock4} for $K = \Ks$. First we resolve the choice of $\Kos$.

\begin{theorem} 
\label{thm:LQCPKLconsis}
Consider the \LQCP and suppose Assumption~\ref{assum:wellposed} holds. 
Then in the state space decomposition \eqref{eq:KalmanDecomp}, $\Kos = \ol{K}_1$, as given in \eqref{eq:Kbar1}.
\end{theorem}

\begin{proof} 
Since $(A,B)$ is stabilizable, 
without loss of generality, $(A,B)$ has the form \eqref{eq:sysctrdecomp}, and $Q$ and $K$ have the block form 
\eqref{eq:QKblock}. Defining $x := (x_1, x_2)$, the Kalman controllability decomposition is 
\begin{align}
\dot{x}_1 &= A_1 x_1 + A_{12} x_2 + B_1 u, \; x_1(0) = x_{1,0} \\
\dot{x}_2 &= A_2 x_2, \; x_2(0) = x_{2,0}.
\end{align}
The controllable subspace is $\cC = \{ x \in \RR^n \; | \; x_2 = 0 \}$. If $x_{2,0} = 0$, then for all $t \geq 0$,
$x_2(t) = 0$ and $x(t) \in \cC$. Thus, we can define a new \LQCPONE on $\cC$ with dynamics 
$\dot{x}_1 = A_1 x_1 + B_1 u$, $x_1(0) = x_{1,0}$, and $(A_1,B_1)$ is controllable. The cost function is 
$J_{1T}(x_{1,0},u) := \int_0^T \omega_1 (x_1(t; x_{1,0},u),u(t)) \, dt$ with 
$\omega_1(x_1,u) := x_1^{\top} Q_1 x_1 + u^{\top} u$. Let $\cL_1 = \cL \cap \cC$ be the terminal subspace 
and let $d_{1 \cL_1} : \RR^{n_1} \rightarrow [0,\infty)$ be the distance function. The input spaces are
\begin{align}
\cU_1(x_{1,0}) &:= 
\left\{\left. u \in L_{2,loc}^m(\RR^+) \; \right| \; 
\lim_{T \rightarrow \infty} J_{1T}(x_{1,0},u) \text{ exists in } \RR^e \right\},\\
\cU_{1\cL_1}(x_{1,0}) &:= 
\left\{\left. u \in \cU_1(x_{1,0}) \; \right| \; 
\lim_{t \rightarrow \infty} d_{1 \cL_1}(x_1(t;x_{1,0},u)) = 0 \right\} \,.
\end{align}
The optimal cost is $V_{1\cL_1}(x_{1,0}) := \inf \{ \lim_{T \rightarrow \infty}J_{1T}(x_{1,0},u) \; | \; u \in \cU_{1\cL_1}(x_{1,0}) \}$. 
The ARE for the \LQCPONE is $\phi_1(K_1) = 0$ as in \eqref{eq:ARE_1} with solution set $\p \Gamma_1$. Consider any 
initial condition $x_0 = (x_{1,0}, 0) \in \cC$ and any control $u \in L_{2,loc}^m(\RR^+)$. Then 
$x(t;x_0,u) = (x_1(t;x_{1,0},u),0)$ and $\omega(x(t;x_0,u),u(t)) = \omega_1(x_1(t;x_{1,0},u),u(t))$, so for all 
$T \ge 0$, $J_T(x_0, u) = J_{1T}(x_{1,0},u)$. Consequently, we have $\cU(x_0) = \cU_1(x_{1,0})$. Also, 
$\lim_{t \rightarrow \infty} d_{\cL}(x(t;x_0,u)) = 0$ is equivalent to $\lim_{t \rightarrow \infty}d_{1 \cL_1}(x_1(t;x_{1,0},u)) = 0$. Thus 
$\cU_{\cL}(x_0) = \cU_{1\cL_1}(x_{1,0})$. With all the above, we conclude that $V_{\cL}(x_0) = V_{1\cL_1} (x_{1,0})$ 
for $x_0 = (x_{1,0},0) \in \cC$. 

Since $(A_1,B_1)$ is controllable, we can apply the results of \cite{TRE89b} to solve the \LQCPONE. 
Since $\Gamma_- \neq \emptyset$, we can apply Lemma \ref{lem:wellposedgivesnegsdefK1} to get that
the minimal solution  $K_1^- \in \p \Gamma_{1-}$ exists. Since $K_1^- \leq 0$, from \eqref{eq:negsemidefquadform} it follows that $K_1^-$ is negative semidefinite on $\cL_1$. By Theorem~\ref{thm:LQCPcontrollable}(ii), $V_{1\cL_1}(x_{1,0}) = x_{1,0}^{\top} \ol{K}_1 x_{1,0}$ with
$\ol{K}_1$ given in \eqref{eq:Kbar1}. Since we have already shown that $V_{\cL}(x_0) = V_{1\cL_1}(x_{1,0})$ for 
$x_0 = (x_{1,0},0) \in \cC$, it can be easily shown that $\Kos = \ol{K}_1$. 
\end{proof}

To resolve the remaining blocks of $\Ks$, we recall some results from \cite{TRE89}. 
For this to apply, we continue to assume that the state space is decomposed according to \eqref{eq:Rnsplit}. It was shown in (5.5) and (5.7) of \cite{TRE89} that $\ol{K}_1$ in \eqref{eq:Kbar1} and the closed-loop system matrix $A_1(\ol{K}_1)$ using $\ol{K}_1$ have the form
\begin{equation}
\label{eq:Kbar1block}
\ol{K}_1 = 
\begin{bmatrix} 
0 & 0                & 0 \\ 
0 & K_{1,22}^-       & K_{1,23}^- \\ 
0 & K_{1,23}^{-\top} & K_{1,33}^+
\end{bmatrix}, \; \;
A_1(\ol{K}_1) = 
\begin{bmatrix}
\ol{A}_{1,11} & 0             & 0 \\ 
0             & \ol{A}_{1,22} & 0 \\ 
0             & 0             & \ol{A}_{1,33}
\end{bmatrix},
\end{equation}
where $\sigma(\ol{A}_{1,11}) \subset \CC^+$, $\sigma(\ol{A}_{1,22}) \subset \CC^0$, and 
$\sigma(\ol{A}_{1,33}) \subset \CC^-$. 
For the choice of $K_1 = \ol{K}_1$ and substituting \eqref{eq:ABblock4}, \eqref{eq:QKblock4}, and \eqref{eq:Kbar1block}, 
the second ARE equation \eqref{eq:ARE_12} splits into three linear Sylvester equations: 
\begin{align}
\ol{A}_{1,11}^{\top} K_{12,1} + K_{12,1} A_2 &= - Q_{12,1},   
\label{eq:ARE_12_1} \\
\ol{A}_{1,22}^{\top} K_{12,2} + K_{12,2} A_2 &= - (Q_{12,2} + K_{1,22}^- A_{12,2} + K_{1,23}^- A_{12,3}), 
\label{eq:ARE_12_2} \\
\ol{A}_{1,33}^{\top} K_{12,3} + K_{12,3} A_2 &= - (Q_{12,3} + K_{1,23}^{- \top} A_{12,2} + K_{1,33}^+ A_{12,3}) 
\label{eq:ARE_12_3} \,.
\end{align}

Using these facts, we can now resolve the remaining blocks of $\Ks$. The main difficulty is that \eqref{eq:ARE_12_1} 
may have an infinite number of solutions for the $K_{12,1}$ block since 
$\sigma(\ol{A}_{1,11}^{\top}) \cap \sigma(-A_2)$ is not 
necessarily empty. The key insight is that $K^{\star}_{12,1}$ can be unambiguously determined by invoking 
Theorem~\ref{thm:attainPrep}(ii) given below, that any negative semidefinite solution $K_N \in \Gamma_-$ 
to the ARI provides a lower bound to the value function. 
In order to utilize this property to resolve the choice of $K_{12,1}^{\star}$, the next lemma describes the 
block structure of any $K_N \in \Gamma_-$.

\begin{lemma} 
\label{lem:KNSD4by4form}
Suppose Assumption~\ref{assum:wellposed} holds and the state space is decomposed as in \eqref{eq:Rnsplit}. Then for all $K_N \in \Gamma_-$, $K_N$ has the block form
\begin{equation*}
K_N = 
\begin{bmatrix} 0 & 0                & 0               & 0 \\ 
                0 & K_{1,22}^-       & K_{1,23}^-      & K_{12,2} \\ 
                0 & K_{1,23}^{-\top} & K_{1,33}        & K_{12,3} \\ 
                0 & K_{12,2}^{\top}  & K_{12,3}^{\top} & K_2 
\end{bmatrix}.
\end{equation*}
\end{lemma}

\begin{proof}
Let $K_N \in \Gamma_-$ have the block form in \eqref{eq:QKblock4}. Since $\Gamma_- \neq \emptyset$ and 
$(A,B)$ is stabilizable, we can apply Lemma~\ref{lem:wellposedgivesnegsdefK1} to obtain that the minimal 
solution $K_1^- \in \p \Gamma_{1 - }$ exists. Also $\p \Gamma_{1 - } \subset \p \Gamma_1 \neq \emptyset$. Because $K_N \in \Gamma_- \subset \Gamma$, by Theorem~\ref{thm:schur} we establish that its upper left block satisfies $K_1 \in \Gamma_1$.
Since $(A_1,B_1)$ is controllable and $\p \Gamma_1 \neq \emptyset$, we can apply Theorem~\ref{thm:extremal}(i) 
to get that the maximal solution $K_1^+ \in \p \Gamma_1$ also exists. Moreover, Theorem~\ref{thm:extremal}(i) also implies that $K_1^-, K_1^+ \in \Gamma_1$, and consequently
$K_1^- \leq K_1 \leq K_1^+$. Since $\p \Gamma_1 \neq \emptyset$, it has been shown (see equation (5.6) in \cite{TRE89}
and equation (5.4) in \cite{TRE89b}) that $K_1^+$, $K_1^-$, and $\Delta_1$ have the block form 
\begin{equation} 
\label{eq:decomp3by3K1plusmin}
K_1^+ = 
\begin{bmatrix} K_{1,11}^+ & 0 & 0 \\ 
                0 & K_{1,22}^+ & K_{1,23}^+ \\ 
                0 & K_{1,23}^{+\top} & K_{1,33}^+ 
\end{bmatrix}, \; \; 
K_1^- = 
\begin{bmatrix} 0 & 0 & 0 \\ 
                0 & K_{1,22}^- & K_{1,23}^- \\ 
                0 & K_{1,23}^{-\top} & K_{1,33}^- 
\end{bmatrix}, \; \; 
\Delta_1 = 
\begin{bmatrix} 
\Delta_{1,11} & 0 & 0 \\ 
0 & 0 & 0 \\ 
0 & 0 & \Delta_{1,33} 
\end{bmatrix} \,,
\end{equation}
where $K_{1,22}^+ = K_{1,22}^-$, $K_{1,23}^+ = K_{1,23}^-$, and $\Delta_{1,33} = K_{1,33}^+ - K_{1,33}^-$. 
Now consider $K_1 \geq K_1^-$ in block form, assuming the decomposition of $K_1^-$ in \eqref{eq:decomp3by3K1plusmin}.
We have 
\begin{equation*}
K_1 - K_1^- = 
\begin{bmatrix}
K_{1,11} - 0 & K_{1,12} - 0 & K_{1,13} - 0 \\
(K_{1,12} - 0)^{\top} & K_{1,22} - K_{1,22}^- & K_{1,23}-K_{1,23}^- \\
(K_{1,13}-0)^{\top} & (K_{1,23} - K_{1,23}^-)^{\top} & K_{1,33}-K_{1,33}^-
\end{bmatrix} \geq 0 \,.
\end{equation*}
Using Theorem~\ref{thm:schur}, we find $K_{1,11} \geq 0$. Since $K_N \in \Gamma_-$ by assumption, $K_N \leq 0$. 
Applying Theorem~\ref{thm:schur} to $K_N = \begin{bmatrix}K_{1,11} & * \\ * & * \end{bmatrix}$, we get 
$K_{1,11} \leq 0$. Thus $K_{1,11} = 0$. Now consider again $K_1^- \leq K_1 \leq K_1^+$ with the information that 
$K_{1,11}= 0$:
\begin{equation*}
\begin{bmatrix} 0 & 0 & 0 \\ 0 & K_{1,22}^- & K_{1,23}^- \\ 0 & (K_{1,23}^-)^{\top} & K_{1,33}^- \end{bmatrix} \leq
\begin{bmatrix} 0 & K_{1,12} & K_{1,13} \\ K_{1,12}^{\top} & K_{1,22} & K_{1,23} \\ K_{1,13}^{\top} & K_{1,23}^{\top} & K_{1,33} \end{bmatrix} \leq
\begin{bmatrix} K_{1,11}^+ & 0 & 0 \\ 0 & K_{1,22}^- & K_{1,23}^- \\ 0 & (K_{1,23}^-)^{\top} & K_{1,33}^+ \end{bmatrix}
\end{equation*}
where we have $K_{1,22}^+ = K_{1,22}^-$ and $K_{1,23}^+ = K_{1,23}^-$ as in \eqref{eq:decomp3by3K1plusmin}. 
We claim that $K_{1,12} = 0$, $K_{1,13} = 0$, $K_{1,22} = K_{1,22}^-$, and $K_{1,23} = K_{1,23}^-$. First, 
we have 
\begin{equation*}
K_1 - K_1^- = \begin{bmatrix} 0 & K_{1,12} & K_{1,13} \\ K_{1,12}^{\top} & * & * \\ K_{1,13}^{\top} & * & * \end{bmatrix} \geq 0.
\end{equation*}
Applying Theorem~\ref{thm:schur} again, we get $(I - 0 0^{\dagger}) \begin{bmatrix} K_{1,12} & K_{1,13} \end{bmatrix} = 0$, 
so that $K_{1,12} = 0$ and $K_{1,13} = 0$. Then $K_1 - K_1^- \geq 0$ reduces to
\begin{equation*}
\begin{bmatrix}
0 & 0 & 0 \\
0 & K_{1,22} - K_{1,22}^- & K_{1,23}-K_{1,23}^- \\
0 & (K_{1,23} - K_{1,23}^-)^{\top} & K_{1,33}-K_{1,33}^-
\end{bmatrix} \geq 0
\end{equation*}
which implies by Theorem \ref{thm:schur} that
\begin{equation*}
\begin{bmatrix}
K_{1,22} - K_{1,22}^- & K_{1,23}-K_{1,23}^- \\
 (K_{1,23} - K_{1,23}^-)^{\top} & K_{1,33}-K_{1,33}^-
\end{bmatrix} \geq 0.
\end{equation*}
Similarly, $K_1^+ - K_1 \geq 0$ gives 
\begin{equation} \label{eq:k1plus-k1lower2x2}
\begin{bmatrix} K_{1,22}^- - K_{1,22} & K_{1,23}^- -K_{1,23} \\
 (K_{1,23}^- - K_{1,23})^{\top} & K_{1,33}^+ -K_{1,33}
\end{bmatrix} \geq 0.
\end{equation}
Applying Theorem \ref{thm:schur} to the previous two statements, we get $K_{1,22}^- \leq K_{1,22} \leq K_{1,22}^-$, 
so $K_{1,22} = K_{1,22}^-$. Then rewriting the previous inequality \eqref{eq:k1plus-k1lower2x2}
\begin{equation*}
\begin{bmatrix} 0 & K_{1,23}^- -K_{1,23} \\
 (K_{1,23}^- - K_{1,23})^{\top} & K_{1,33}^+ -K_{1,33}
\end{bmatrix} \geq 0.
\end{equation*}
Applying Theorem \ref{thm:schur}, we get $(I - 0 0^{\dagger})(K_{1,23}^- - K_{1,23}) = 0$, 
so $K_{1,23} = K_{1,23}^-$. So far we have for $K_N \in \Gamma_-$
\begin{equation*}
K_N = \begin{bmatrix} 
0 & 0 & 0 & K_{12,1} \\
0 & K_{1,22}^- & K_{1,23}^- & K_{12,2} \\
0 & (K_{1,23}^-)^{\top} & K_{1,33} & K_{12,3} \\
K_{12,1}^{\top} & K_{12,2}^{\top} & K_{12,3}^{\top} & K_2
\end{bmatrix} \leq 0.
\end{equation*}
Then $-K_N \geq 0$ has the block form:
\begin{equation*}
\begin{bmatrix} 0 & 0 & -K_{12,1} \\ 0 & * & * \\ -K_{12,1}^{\top} & * & * \end{bmatrix} \geq 0.
\end{equation*}
Applying Lemma~\ref{lem:3by3blockschur}, we get $K_{12,1} = 0$.
\end{proof}

In the next result we completely characterize the form of $\Ks$.
Before proceeding with this result, we collect some well known results about the cost function.  

\begin{theorem}
\label{thm:attainPrep}
Consider the system \eqref{eq:sysgeneral} with the cost function 
\eqref{eq:finTimeCostFunctGeneral} - \eqref{eq:qformMostGeneral}. Let $x_0 \in \RR^n$, $T \geq 0$, and
$u \in L_{2,loc}^m(\RR^+)$. 
\begin{enumerate}
\item[(i)]
Let $K \in \partial \Gamma$. Then 
$J_T(x_0,u) = \int_0^T \| u(t) + B^{\top}Kx(t)\|^2 \, dt + x_0^{\top}Kx_0 - x^{\top}(T)Kx(T)$,
where $x(t) := x(t;x_0,u)$.
\item[(ii)] 
For all $x_0 \in \RR^n$ and $K_N \in \Gamma_-$, $V_{\cL}(x_0) \ge x_0^{\top} K_N x_0$.
\item[(iii)] 
Suppose Assumption~\ref{assum:wellposed} holds. 
If $J(x_0,u) = x_0^{\top} \Ks x_0$, then $u = -B^{\top} \Ks x$ and 
$\lim_{T \rightarrow \infty} x^{\top}(T) \Ks x(T) = 0$.
\end{enumerate}
%Let $u \in \cU_{\cL}(x_0)$. Then
%$J(x_0,u) \geq \int_0^{\infty} \| u(t) + B^{\top} \Ks x(t) \|^2 \, dt + x_0^{\top} \Ks x_0$.
\end{theorem}

\begin{proof}
Statement (i) is standard. See for instance \cite{WIL71} or \cite{TRE89}.
Statement (ii) is Proposition 1.8 of \cite{GEE89}. See also Lemma 4.4 of \cite{TRE89}.
Statement (iii) is Theorem 2.8(c) of \cite{GEE89}. See also the proof of Theorem 5.1(iii) in \cite{TRE89}.
\end{proof}

\begin{theorem} 
\label{thm:Ks}
Consider the \LQCP and suppose Assumption~\ref{assum:wellposed} holds. 
Then in the state space decomposition \eqref{eq:Rnsplit}, $\Ks \in \p \Gamma$ has the form
\begin{equation}
\label{eq:Ksblock}
\Ks = 
\begin{bmatrix} 0 & 0                         & 0                         & 0 \\ 
                0 & K_{1,22}^-                & K_{1,23}^-                & K_{12,2}^{\star} \\ 
                0 & K_{1,23}^{-\top}          & K_{1,33}^+                & K_{12,3}^{\star} \\
                0 & (K_{12,2}^{\star})^{\top} & (K_{12,3}^{\star})^{\top} & K_2^{\star} 
\end{bmatrix} \,,
\end{equation}
where $K_{12,2}^{\star}$ is the unique solution to \eqref{eq:ARE_12_2}, $K_{12,3}^{\star}$ is the unique solution 
to \eqref{eq:ARE_12_3}, and $K_2^{\star}$ is the unique solution to \eqref{eq:ARE_2} with $K_{12} = K_{12}^{\star}$. 
\end{theorem}

\begin{proof}
By Theorem \ref{thm:LQCPKLconsis}, $\Kos = \ol{K}_1$ with the form of $\ol{K}_1$ given in \eqref{eq:Kbar1}. 
By Theorem~\ref{thm:valuefunction}, $\Ks \in \p \Gamma$. 
Next we consider \eqref{eq:ARE_12}. Using the decompositions above and with the choice $K_1 = \ol{K}_1$, 
the second ARE equation \eqref{eq:ARE_12} splits into \eqref{eq:ARE_12_1}, \eqref{eq:ARE_12_2}, and \eqref{eq:ARE_12_3}.
Since $\sigma(\ol{A}_{1,22}) \subset \CC^0$, $\sigma(\ol{A}_{1,33}) \subset \CC^-$, and $\sigma(-A_2) \subset \CC^+$,
\eqref{eq:ARE_12_2} and \eqref{eq:ARE_12_3} have unique solutions $K_{12,2}^{\star}$ and $K_{12,3}^{\star}$, 
respectively \cite{GANTMACHER}. Similarly, \eqref{eq:ARE_2} has a unique solution $K_2^{\star}$, assuming 
$K_{12} = K_{12}^{\star}$. At this point we know that $\Ks$ has the block form:
\begin{equation}
\label{eq:Ks2}
\Ks = 
\begin{bmatrix} 0                         & 0                         & 0                         & K_{12,1}^{\star} \\ 
                0                         & K_{1,22}^-                & K_{1,23}^-                & K_{12,2}^{\star} \\ 
                0                         & K_{1,23}^{-\top}          & K_{1,33}^+                & K_{12,3}^{\star} \\
                (K_{12,1}^{\star})^{\top} & (K_{12,2}^{\star})^{\top} & (K_{12,3}^{\star})^{\top} & K_2^{\star} 
\end{bmatrix} \,.
\end{equation}
Comparing to \eqref{eq:Ksblock}, it remains only to show that $K_{12,1}^{\star} = 0$. 
By Theorem~\ref{thm:valuefunction}, $V_{\cL}(x_0) = x_0^{\top} \Ks x_0$. Let $K_N \in \Gamma_-$. 
By Theorem~\ref{thm:attainPrep}(ii), for all $x_0 \in \RR^n$, $V_{\cL}(x_0) = x_0^{\top} \Ks x_0 \ge x_0^{\top} K_N x_0$;
that is, $\Ks \geq K_N$. Using the block form of $\Ks$ in \eqref{eq:Ks2} and the block form of $K_N$ in 
Lemma~\ref{lem:KNSD4by4form}, we have
\begin{equation*}
\Ks - K_N = 
\begin{bmatrix}
0                         & 0                                             & 0 
& K_{12,1}^{\star} \\ 
0                         & 0                                             & 0 
& K_{12,2}^{\star} - K_{12,2 N} \\
0                         & 0                                             & K_{1,33}^+ - K_{1,33 N} 
& K_{12,3}^{\star} - K_{12,3 N} \\
(K_{12,1}^{\star})^{\top} & (K_{12,2}^{\star} - K_{12,2 N})^{\top} & (K_{12,3}^{\star} - K_{12,3 N})^{\top} 
& K_2^{\star} - K_{2 N} 
\end{bmatrix} \geq 0 \,.
\end{equation*}
Applying Lemma~\ref{lem:3by3blockschur} yields that $K_{12,1}^{\star} = 0$, as desired. 
\end{proof}                       

\begin{remark}
\label{rem:Q}
We observe from the form of $\Ks$ that $K^{\star}_{12,1} = 0$. If we substitute $K^{\star}_{12,1} = 0$ into 
\eqref{eq:ARE_12_1}, we get that $Q_{12,1} = 0$. One can derive the fact that $Q_{12,1} = 0$ via a separate argument,
and this provides an independent validation of our result that $K^{\star}_{12,1} = 0$. Suppose 
Assumption~\ref{assum:wellposed} holds. Take any symmetric $K$ with the special form:
\[
K = 
\begin{bmatrix} 0 & 0 & 0 & 0 \\ 
                0 & K_{1,22} & K_{1,23} & K_{12,2} \\ 
                0 & K_{1,23}^{\top} & K_{1,33} & K_{12,3} \\ 
                0 & K_{12,2}^{\top} & K_{12,3}^{\top} & K_2 
\end{bmatrix} \,.
\]
We decompose $A$ and $B$ as in \eqref{eq:ABblock4}. Using a result analogous to equation (5.2) in \cite{TRE89}, 
it can be shown that $\cN_1(\cL_1)$ is $A_1$-invariant, and this implies $A_{1,21} = A_{1,31} = 0$. 
Then by direct computation $\phi(K)$ has the form: 
\[
\phi(K) = 
\begin{bmatrix} Q_{1,11} & Q_{1,12} & Q_{1,13} & Q_{12,1} \\
                Q_{1,12}^{\top} & * & * & * \\
                Q_{1,13}^{\top} & * & * & * \\
                Q_{12,1}^{\top} & * & * & * \end{bmatrix} \,.
\]
Now choose the upper left block of the above $K$ to be $K_1^- \in \p \Gamma$. By \eqref{eq:decomp3by3K1plusmin} this
choice is consistent with the form of $K$ above. Since the upper left block of $\phi(K)$ is written as $\phi_1(K_1)$  
and we know that $\phi_1(K_1^-) = 0$, it immediately follows that $Q_{1,11} = 0$, $Q_{1,12} = 0$, and $Q_{1,13} = 0$. 
Next, since $\Gamma_- \neq \emptyset$, let $K_N \in \Gamma_-$. By Lemma \ref{lem:KNSD4by4form}, $K_N$ has the special 
form above. Then we have 
\begin{equation*}
\phi(K_N) = 
\begin{bmatrix} 0 & 0 & 0 & Q_{12,1} \\ 
                0 & * & * & *        \\ 
                0 & * & * & * \\
                Q_{12,1}^{\top} & * & * & * 
\end{bmatrix} \geq 0.
\end{equation*}
By applying Lemma~\ref{lem:3by3blockschur}, we conclude that $Q_{12,1} = 0$. 
\end{remark} 

We conclude this section by applying Theorem \ref{thm:Ks} to obtain necessary and sufficient conditions for 
attainability of the \LQCP. Remarkably, the attainability result depends only on the controllable subspace. 

\begin{theorem}
\label{thm:attainable}
Suppose Assumption~\ref{assum:wellposed} holds and the state space is decomposed as in \eqref{eq:KalmanDecomp}.
Then the \LQCP is attainable if and only if $\Ker(\Delta_1) \subset \cL_1 \cap \Ker(K_1^-)$.
\end{theorem}

\begin{proof}
Due to Assumption~\ref{assum:wellposed}, we may further assume that the state space is decomposed according to \eqref{eq:Rnsplit}. Let $W_1 \in \RR^{n_1 \times n_1}$ be a matrix such that $\Ker (W_1)  = \cL_1$ and let 
$d_{1 \cL_1}: \RR^{n_1} \rightarrow [0,\infty)$ be the distance function in $\RR^{n_1}$ to $\cL_1$. 
Since $\cX_{1,1} = \langle \cL_1 \cap \Ker (K_1^-) \; | \; A_1(K_1^-) \rangle \cap \cX^+ (A_1(K_1^-))$,
we have $\cX_{1,1} \subset \langle \Ker (W_1) \cap \Ker (K_1^-) \; | \; A_1(K_1^-) \rangle \subset
\Ker (W_1) \cap \Ker (K_1^-) = \Ker \left( \begin{bmatrix}K_1^- \\ W_1 \end{bmatrix} \right)$. We claim 
\begin{equation}
\label{eq:D}
\begin{bmatrix} K_1^- \\ W_1 \end{bmatrix} = \begin{bmatrix} 0 & D_2 & D_3 \end{bmatrix} \,.
\end{equation}
Proof of Claim: Let $x_1 \in \cX_{1,1}$. Then $x_1 \in \Ker \left( \begin{bmatrix}K_1^- \\ W_1 \end{bmatrix} \right) =: 
\Ker \left( \begin{bmatrix} D_1 & D_2 & D_3 \end{bmatrix} \right)$. Also since $x_1 \in \cX_{1,1}$, in coordinates it has the form
$x_1 = (x_{1,1},0,0)$. Then $\begin{bmatrix} D_1 & D_2 & D_3 \end{bmatrix} x_1 = D_1 x_{1,1} = 0$. Since
$x_{1,1}$ is arbitrary, we get $D_1 = 0$, as desired.  

$(\Rightarrow)$ \
Suppose the \LQCP is attainable. Let $x_0 \in \RR^n$. By definition there exists $u^{\star} \in \cU_{\cL}(x_0)$ such
that $V_{\cL}(x_0) = J(x_0,u^{\star})$. By Theorem~\ref{thm:valuefunction} we know $V_{\cL}(x_0) = x_0^{\top} \Ks x_0$ where
$\Ks \in \p \Gamma$, and by Theorem~\ref{thm:Ks}, $\Ks$ is given in \eqref{eq:Ksblock}. Now we can apply 
Theorem~\ref{thm:attainPrep}(iii) to get $u^{\star} = -B^{\top} \Ks x$. The closed-loop
dynamics are $\dot{x} = A(\Ks) x$. Let $x := (x_1, x_2) := (x_{1,1},x_{1,2}, x_{1,3}, x_2 )$
according to the decomposition \eqref{eq:Rnsplit}. Then using the block form of $A_1(\Kos) = A(\ol{K}_1)$ in 
\eqref{eq:Kbar1block}, we have 
\begin{equation} 
\label{eq:optimalclsdyn}
\dot{x} = 
\begin{bmatrix} 
A_1(\Kos) & * \\ 0 & A_2 \end{bmatrix}
\begin{bmatrix} x_1 \\ x_2 \end{bmatrix} = 
\begin{bmatrix} 
\ol{A}_{1,11} & 0             & 0             & \ol{A}_{12,1} \\
            0 & \ol{A}_{1,22} & 0             & \ol{A}_{12,2} \\
            0 & 0             & \ol{A}_{1,33} & \ol{A}_{12,3} \\
            0 & 0             & 0             & A_2 
\end{bmatrix}
\begin{bmatrix} x_{1,1} \\ x_{1,2} \\ x_{1,3} \\ x_2 \end{bmatrix} \,,
\end{equation}
where $\sigma(\ol{A}_{1,11}) \subset \CC^+$, $\sigma(\ol{A}_{1,22}) \subset \CC^0$, $\sigma(\ol{A}_{1,33}) \subset \CC^-$,
and by stabilizability, $\sigma(A_2) \subset \CC^-$. Using the variation of constants formula we get that at $t = T$
\begin{equation}  
\label{eq:attainvarconsts}
x_{1,i}(T) = \textrm{e}^{\ol{A}_{1,ii} T} x_{1,i}(0) + 
             \int_0^T \textrm{e}^{\ol{A}_{1,ii} (T-\tau)} \ol{A}_{12,i} \textrm{e}^{A_2 \tau} x_2(0) \, d\tau, \; \; 
i = 1,2,3 \,.
\end{equation}
Since $\sigma(A_2) \subset \CC^-$, $\lim_{T \rightarrow \infty} x_2(T) = 0$. Using \eqref{eq:attainvarconsts} for $i = 3$, 
$\sigma(\ol{A}_{1,33}) \subset \CC^-$, and the fact that $\lim_{T \rightarrow \infty}x_2(T) = 0$, we also get 
$\lim_{T \rightarrow \infty}x_{1,3}(T) = 0$. Now using \eqref{eq:Ksblock}, the block form of $K_1^-$ given in 
\eqref{eq:decomp3by3K1plusmin}, and the fact that $K_{1,33}^+ = \Delta_{1,33} + K_{1,33}^-$, we have
\begin{align}
x^{\top} \Ks x 
&= x_1^{\top} \Kos x_1 + 2 x_1^{\top} K_{12}^{\star} x_2 + x_2^{\top} K_2^{\star} x_2 \nonumber \\
&= x_1^{\top} K_1^- x_1 + x_{1,3}^{\top} \Delta_{1,33} x_{1,3} + 2(x_{1,2}^{\top} K_{12,2}^{\star} 
 + x_{1,3}^{\top} K_{12,3}^{\star}) x_2 + x_2^{\top} K_2^{\star} x_2 \,. \label{eq:xKsx}
\end{align}
Using this expression combined with the fact that $\lim_{T \rightarrow \infty}x_{1,3}(T) = 0$, $\lim_{T \rightarrow \infty}x_2(T) = 0$, and 
$\lim_{T \rightarrow \infty}x^{\top}(T) \Ks x(T) = 0$ from Theorem~\ref{thm:attainPrep}(iii), we get
\begin{equation}
\label{eq:both}
\lim_{T \rightarrow \infty} x^{\top}(T) \Ks x(T) = 
\lim_{T \rightarrow \infty} \left( x_1^{\top}(T) K_1^-  x_1(T) + 2x_{1,2}^{\top}(T) K_{12,2}^{\star} x_2(T) \right) = 0 \,. 
\end{equation}
Now we observe that $\lim_{T \rightarrow \infty} 2x_{1,2}^{\top}(T) K_{12,2}^{\star} x_2(T) = 0$ because
$\sigma(\ol{A}_{1,22}) \subset \CC^0$ and $\sigma(A_2) \subset \CC^-$.
Returning to \eqref{eq:both}, this implies that also $\lim_{T \rightarrow \infty} x_1^{\top}(T) K_1^-  x_1(T) = 0$. 

We have assumed that $u^{\star} \in \cU_{\cL}(x_0)$ and $(A,B)$ is stabilizable. Therefore, 
$\lim_{T \rightarrow \infty}d_{\cL \cap \cC}(x(T)) = 0$, and thus within the controllable subspace $\lim_{T \rightarrow \infty}d_{1 \cL_1}(x_1(T)) = 0$. Since $\cL_1 = \Ker (W_1)$, 
$\lim_{T \rightarrow \infty}W_1 x_1(T) = 0$. Meanwhile by Lemma~\ref{lem:wellposedgivesnegsdefK1}, $K_1^- \le 0$. Since 
$\lim_{T \rightarrow \infty}x_1^{\top}(T) K_1^-  x_1(T) = 0$, by taking the limit in \eqref{eq:negsemidefquadform} we have that $\lim_{T \rightarrow \infty}K_1^- x_1(T) = 0$. 
Overall, we have $\lim_{T \rightarrow \infty} \begin{bmatrix} K_1^- \\ W_1 \end{bmatrix} x_1(T) = 0$. Using \eqref{eq:D}, this gives
$\lim_{T \rightarrow \infty} \left( D_2 x_{1,2}(T) + D_3 x_{1,3}(T) \right) = 0$. We already know that $\lim_{T \rightarrow \infty}x_{1,3}(T) = 0$, so we get
$\lim_{T \rightarrow \infty}D_2 x_{1,2}(T) = 0$. However, $\sigma(\ol{A}_{1,22}) \subset \CC^0$ and $x_{1,2}(0)$ is arbitrary, so
$D_2 = 0$. Finally, we observe that if $x_1 \in \cX_{1,2}$, then $\begin{bmatrix}0 & 0 & D_3 \end{bmatrix} x_1 = 0$
since $x_1 = (0,x_{1,2},0)$. That is, $\cX_{1,2} \subset \Ker \left( \begin{bmatrix}0 & 0 & D_3 \end{bmatrix} \right)$.
In sum, we have 
\begin{equation} 
\label{eq:kertriangleInL1K1}
\cX_{1,2} = \Ker (\Delta_1) \subset \Ker \left( \begin{bmatrix}0 & 0 & D_3 \end{bmatrix} \right) = 
            \Ker \left( \begin{bmatrix} K_1^- \\ W_1 \end{bmatrix} \right) = \cL_1 \cap \Ker (K_1^-) \,.
\end{equation}

$(\Leftarrow)$ \ 
Suppose that $\Ker (\Delta_1) \subset \cL_1 \cap \Ker (K_1^-)$. Let $x_0 \in \RR^n$. To show attainability, we must find an
optimal control. Consider the candidate $u^c := -B^{\top} \Ks x$, where $\Ks$ is given in \eqref{eq:Ksblock}. We must 
show $V_{\cL}(x_0) = J(x_0,u^c)$ and $u^c \in \cU_{\cL}(x_0)$. The closed-loop dynamics using $u^c$ are given in 
\eqref{eq:optimalclsdyn}. Following the same arguments as above we have that $\lim_{T \rightarrow \infty}x_2(T) = 0$ and 
$\lim_{T \rightarrow \infty}x_{1,3}(T) = 0$. By assumption, $\Ker (\Delta_1 )\subset \cL_1 \cap \Ker (K_1^-)$. From above,
$\cL_1 \cap \Ker (K_1^-) = \Ker \left( \begin{bmatrix} K_1^- \\ W_1 \end{bmatrix} \right) = \Ker \left( \begin{bmatrix}0 & D_2 & D_3 \end{bmatrix} \right) $.
We claim that $D_2 = 0$. To see this, let $x_1 \in \Ker (\Delta_1) = \cX_{1,2}$. Then $x_1 = (0,x_{1,2}, 0)$. 
Since $\Ker (\Delta_1) \subset \Ker \left( \begin{bmatrix}0 & D_2 & D_3 \end{bmatrix} \right)$ we have 
$\begin{bmatrix}0 & D_2 & D_3 \end{bmatrix} x_1 = D_2 x_{1,2} = 0$. Since $x_{1,2}$ is arbitrary, $D_2 = 0$.  
Using the block form of $K_1^-$ in \eqref{eq:decomp3by3K1plusmin}, we have
\[
\begin{bmatrix} K_1^- \\ W_1 \end{bmatrix} =
\begin{bmatrix} 
0 & 0 & 0 \\ 
0 & K_{1,22}^- & K_{1,23}^- \\ 
0 & K_{1,23}^{-\top} & K_{1,33}^- \\
W_{11} & W_{12} & W_{13}  
\end{bmatrix} =
\begin{bmatrix}0 & 0 & D_3 \end{bmatrix} \,.
\]
This implies $K_{1,22}^- = K_{1,23}^- = 0$.
Now we observe $\Ks \in \p \Gamma$ by Theorem~\ref{thm:valuefunction} and $u^c \in L_{2,loc}^{m}(\RR^+)$ for any
fixed $T \ge 0$. Therefore, we can apply Theorem~\ref{thm:attainPrep}(i) with $K = \Ks$ and $u = u^c$ to get
\begin{align}
\label{eq:JTc}
J_T(x_0,u^c) &= x_0^{\top} \Ks x_0 - x(T)^{\top} \Ks x(T) \,. 
\end{align}
We claim that $\lim_{T \rightarrow \infty} x^{\top}(T) \Ks x(T) = 0$. Using the expansion of $x(T)^{\top} \Ks x(T)$
given in \eqref{eq:xKsx}, and the fact that $\lim_{T \rightarrow \infty}x_2(T) = 0$ and $\lim_{T \rightarrow \infty}x_{1,3}(T) = 0$, we get 
$\lim_{T \rightarrow \infty} x^{\top}(T) \Ks x(T) = \lim_{T \rightarrow \infty} x_1^{\top}(T) K_1^- x_1(T)$. 
Using the available information about the block form of $K_1^-$ and that $\lim_{T \rightarrow \infty}x_{1,3}(T) = 0$, we find
\[
\lim_{T \rightarrow \infty} x^{\top}(T) \Ks x(T) = 
\lim_{T \rightarrow \infty} x_1^{\top}(T) 
\begin{bmatrix}
0 & 0 & 0 \\
0 & 0 & 0 \\
0 & 0 & K_{1,33}^- \\
\end{bmatrix} x_1(T)= \lim_{T \rightarrow \infty} x_{1,3}^{\top}(T) K_{1,33}^- x_{1,3}(T) = 0 \,.
\]
Returning to \eqref{eq:JTc}, we have $\lim_{T \rightarrow \infty} J_T(x_0,u^c) = J(x_0,u^c) = x_0^{\top} \Ks x_0$,
as desired. 

Finally, we must show $u^c \in \cU_{\cL}(x_0)$, and particularly $\lim_{T \rightarrow \infty}d_{\cL}(x(T)) = 0$.
Since $\lim_{T \rightarrow \infty}x_{1,3}(T) = 0$, we have that
\[
\lim_{T \rightarrow \infty} \begin{bmatrix} K_1^- \\ W_1 \end{bmatrix} x_1(T) = 
\lim_{T \rightarrow \infty} \begin{bmatrix} 0 & 0 & D_3 \end{bmatrix} x_1(T) =
D_3 x_{1,3}(T) = 0. 
\]
Thus, $\lim_{T \rightarrow \infty}W_1 x_1(T) = 0$, which implies $\lim_{T \rightarrow \infty}d_{1 \cL_1}(x_1(T)) = 0$. Since $\cL_1 = \cL \cap \cC$ and 
$\lim_{T \rightarrow \infty}x_2(T) = 0$, we have $\lim_{T \rightarrow \infty}d_{\cL} (x(T)) = 0$. Thus, $u^c \in \cU_{\cL} (x_0)$, as desired.
\end{proof}

We collect all of the previous results to obtain the culminating result on the solution of the \LQCP. 
It is a generalization of Theorem \ref{thm:LQCPcontrollable} for the case of $(A,B)$ controllable to the case
when $(A,B)$ is stabilizable.

\begin{theorem} 
\label{thm:mainresult}
Consider the \LQCP. Suppose Assumption~\ref{assum:wellposed} holds and the state space is decomposed as in \eqref{eq:KalmanDecomp}. Then we have
\begin{enumerate}
\item[(i)] 
The problem is well-posed.
\item[(ii)] 
For all $x_0 \in \RR^n$, $V_{\cL}(x_0) = x_0^{\top} \Ks x_0$. 
\item[(iii)] 
For all $x_0 \in \RR^n$, the problem is attainable if and only if $\Ker(\Delta_1) \subset \cL_1 \cap \Ker(K_1^-)$.
\item[(iv)] 
If the problem is attainable, then for each $x_0 \in \RR^n$, there exists exactly one optimal 
input $u^{\star}$, and it is given by $u^{\star} = -B^{\top} \Ks x$.
\end{enumerate}
\end{theorem}

\begin{proof}
Statements (i) and (ii) follow from Theorem~\ref{thm:valuefunction}. The form of $\Ks$ follows from Theorem~\ref{thm:Ks}.
Statement (iii) is an immediate consequence of Theorem~\ref{thm:attainable}. 
Statement (iv) follows from Theorem \ref{thm:attainPrep} (iii).
\end{proof}

\section{Discussion}
\label{sec:discuss}

In this section we discuss several special cases of our main result. This includes a comparison with classical 
results in the positive semidefinite case.
First, we consider the special case when $\cN_1(\cL_1) = 0$ which was also treated in Theorem 6.1 of \cite{TRE89}. 
From our experience it is only in exceptional cases that $\cN_1(\cL_1) \neq 0$. The following result shows that
when $\cN_1(\cL_1) = 0$, then $\Ks = K^+$, the maximal solution in $\p \Gamma$.
This result has practical significance because there are many powerful algorithms for numerically 
finding the maximal solution of the ARE. 

\begin{theorem}
\label{thm:N1L1Zero}
Consider the \LQCP, suppose that Assumption~\ref{assum:wellposed} holds, and that the state space decomposed as in \eqref{eq:KalmanDecomp}. Then 
$\cN_1(\cL_1) = 0$ if and only if $K^* = K^+$, where $K^+ \in \p \Gamma$ is the maximal solution.
\end{theorem}

\begin{proof}
(Only if) \
Suppose $\cN_1(\cL_1) = 0$. By Theorem~\ref{thm:Ks},  
$\Ks := \begin{bmatrix} K^{\star}_1 & K^{\star}_{12} \\ K^{\star \top}_{12} & K^{\star}_2 \end{bmatrix} \in \p \Gamma$,
where $K_1^{\star} = \ol{K}_1 = \gamma(\cN_1(\cL_1))$. By assumption, $P_{\cN_1(\cL_1)} = 0$, and
then \eqref{eq:Kbar1} gives $K_1^{\star} = K_1^+$, where $K_1^+$ is the maximal solution in $\p \Gamma_1$.
By Theorem~\ref{thm:extremal}(ii), we also know $K_1^+ \in \p \Gamma_1$ is the unique maximal solution such
that $\sigma(A_1(K_1^+)) \subset \CC^- \cup \CC^0$. Furthermore, by stabilizability, $\sigma(A_2) \subset \CC^-$.
Therefore, $\sigma(A_1(K_1^+)) \cap \sigma(-A_2) = \emptyset$, so $K_{12}^{\star}$ is the unique solution of the 
Sylvester equation \eqref{eq:ARE_12}. Similarly, since $\sigma(A_2^{\top}) \cap \sigma(-A_2) = \emptyset$, $K_2^{\star}$ is the
unique solution of the Sylvester equation \eqref{eq:ARE_2}.

Meanwhile, since $\p \Gamma \neq \emptyset$, by Theorem~\ref{thm:gohberg}, the maximal solution $K^+ \in \p \Gamma$ 
exists and satisfies $\sigma(A(K^+)) \subset \CC^- \cup \CC^0$. We claim $K^{\star} = K^+$. Let 
$K^+ = \begin{bmatrix} K_1 & K_{12} \\ K^{\top}_{12} & K_2 \end{bmatrix}$ in block form. 
Since $K^+ \in \p \Gamma$, we have that $K_1 \in \p \Gamma_1$. Using \eqref{eq:sysctrdecomp}, 
$\sigma(A(K^+)) = \sigma(A_1(K_1)) \uplus \sigma(A_2) \subset \CC^- \cup \CC^0$. Then since 
$\sigma(A_2) \subset \CC^-$, we have $\sigma(A_1(K_1)) \subset \CC^- \cup \CC^0$.
However, by Theorem~\ref{thm:extremal}(ii), $K_1 \in \p \Gamma_1$ and 
$\sigma(A_1(K_1)) \subset \CC^- \cup \CC^0$ together imply $K_1 = K_1^+ = K_1^{\star}$, 
the unique maximal solution in $\p \Gamma_1$. It immediately follows that $K_{12} = K^{\star}_{12}$ and
$K_2 = K^{\star}_2$, as desired. 

(If) \
Suppose $K^* = K^+$, the maximal solution in $\p \Gamma$. By writing $K^+$ in block form,  
$K^+ = \begin{bmatrix} K_1 & K_{12} \\ K_{12}^{\top} & K_2 \end{bmatrix}$, we have $K_1 = \Kos$. We also have 
that $K_1 = K_1^+$ is the maximal solution in $\p \Gamma_1$ using an argument analogous to the one above. 
That is, using \eqref{eq:sysctrdecomp}, $\sigma (A(K^+)) = \sigma(A_1(K_1)) \uplus \sigma(A_2)$. 
By Theorem~\ref{thm:gohberg}, $\sigma (A(K^+)) \subset \CC^- \cup \CC^0$. Since $\sigma(A_2) \subset \CC^-$, 
we get $\sigma(A_1(K_1)) \subset \CC^- \cup \CC^0$. Then by Theorem~\ref{thm:extremal}(ii), 
$K_1 = K_1^+ \in \p \Gamma_1$. 
Meanwhile by Theorem \ref{thm:LQCPKLconsis}, $\ol{K}_1 = \Kos$. Putting this altogether, we have that 
$\ol{K}_1 = \Kos = K_1^+$. Finally, using $\ol{K}_1 = K_1^+$ in \eqref{eq:Kbar1} gives that $P_{\cN_1(\cL_1)} = 0$, 
so $\cN_1(\cL_1) = 0$. 
\end{proof}

Next we discuss how Theorem~\ref{thm:mainresult} recovers well known results for the free-endpoint and fixed-endpoint 
problems when $Q$ is positive semidefinite and $(A,B)$ is stabilizable. First, we observe that when $Q \geq 0$, then 
$\phi(0) \geq 0$ so $0 \in \Gamma_- \neq \emptyset$. Therefore, Assumption~\ref{assum:wellposed} holds. We also assume that the state space is decomposed as in \eqref{eq:Rnsplit} wherever needed.

The main results on the free endpoint problem are summarized in Theorem~10.13 in \cite{TRE02}. In particular, when $\cL = \RR^n$,
$V_{\cL}(x_0) = x_0^{\top} P^- x_0$, where $P^- \geq 0$ is the smallest positive semidefinite solution to the ARE, 
and the optimal control is $u^{\star}(t) = - B^{\top} P^- x(t)$. We would like to verify that our 
Theorem~\ref{thm:mainresult} recovers these results. We will show that when $Q \ge 0$, $\Ks$ given 
in \eqref{eq:Ksblock} satisfies $\Ks = P^-$. 
To aid in this endeavor, we invoke a result from \cite{TRE89}. Let 
$\p \Gamma_{1 +} := \{K_1 \in \p \Gamma_1  ~|~  K_1 \ge 0 \}$. 

\begin{theorem}[Theorem 6.3 \cite{TRE89}]
\label{thm:6.3}
Assume $(A_1,B_1)$ is controllable and $\p \Gamma_{1 - } \neq \emptyset$. Then the following hold: if 
$\p \Gamma_{1 +} \neq \emptyset$, then (i) $\ol{K}_1 \in \p \Gamma_{1 +}$ and (ii) $K_1 \in \p \Gamma_{1 +}$ 
implies $\ol{K}_1 \leq K_1$.
\end{theorem}

\begin{lemma}
Consider the \LQCP. Suppose $(A,B)$ is stabilizable, $\cL = \RR^n$, and
$Q \ge 0$. Then $\Ks = P^-$.
\end{lemma}

\begin{proof}
We begin by applying Theorem~\ref{thm:6.3} to show that $\ol{K}_1$ is the smallest solution in $\p \Gamma_{1 +}$. To that
end, we must show that $\p \Gamma_{1 - } \neq \emptyset$ and $\p \Gamma_{1 +} \neq \emptyset$. First, since 
Assumption~\ref{assum:wellposed} holds, we can apply Lemma~\ref{lem:wellposedgivesnegsdefK1} to get 
$K_1^- \in \p \Gamma_{1 - }$ exists, so $\p \Gamma_{1 - } \neq \emptyset$. Second, because $Q \ge 0$, we know
$\phi(0) \ge 0$, so $0 \in \Gamma_-$. By Theorem~\ref{thm:valuefunction}, $V_{\cL}(x) = x^{\top} \Ks x$. Applying 
Theorem~\ref{thm:attainPrep}(ii) with $K_N = 0$, we get $x^{\top} \Ks x \ge x^{\top} 0 x = 0$, for all $x \in \RR^n$,
so $\Ks \ge 0$. That is, $\Ks \in \p \Gamma_+$. By Theorem~\ref{thm:schur}, this implies 
$K_1^{\star} \ge 0$, so $K_1^{\star} = \ol{K}_1 \in \p \Gamma_{1+} \neq 0$. Now we can apply Theorem~\ref{thm:6.3} to
get $K_1^{\star} = \ol{K}_1$ is the smallest solution in $\p \Gamma_{1+}$. 

It remains to show that $\Ks = P^-$ is the smallest solution in $\p \Gamma_+$. 
To arrive at a contradiction, suppose there 
exists $K \in \p \Gamma_+$ such that $K \neq \Ks$ and $K \le \Ks$. There are two cases. First, suppose $K \in \p \Gamma_+$
with $K \leq \Ks$ such that $K_1 \neq \Kos$, where $K_1$ is the upper left block of $K$. Since $K \in \p \Gamma$,
$\phi(K) = 0$, so $\phi_1(K_1) = 0$, implying $K_1 \in \p \Gamma_1$. By Theorem~\ref{thm:schur}, $K \ge 0$ implies 
$K_1 \ge 0$, so $K_1 \in \p \Gamma_{1+}$. Again by Theorem~\ref{thm:schur}, $K \le \Ks$ implies $K_1 \le \Kos$.
Thus, we have $K_1 \in \p \Gamma_{1+}$ such that $K_1 \le \Kos$, which contradicts that $\Kos$ is the smallest 
solution in $\p \Gamma_{1+}$.

For the second case,  suppose $K \in \p \Gamma_+$ with $K \leq \Ks$ such that $K_1 = \Kos$. By \eqref{eq:Kbar1block}, 
$K$ has the form
\[
K = \begin{bmatrix} 0 & 0                         & 0                         & K_{12,1} \\ 
                0 & K_{1,22}^-                & K_{1,23}^-                & K_{12,2} \\ 
                0 & K_{1,23}^{-\top}          & K_{1,33}^+                & K_{12,3} \\
                K_{12,1}^{\top} & K_{12,2}^{\top} & K_{12,3}^{\top} & K_2
\end{bmatrix} \,.
\]
Since $K \geq 0$, we can apply Lemma~\ref{lem:3by3blockschur} to find that $K_{12,1} = 0$. Then since $K_1 = \Kos$,
$K_{12,1} = K_{12,1}^{\star} = 0$, and $\phi(K) = 0$, the solutions for $K_{12,2}$ and $K_{12,3}$ are unique and 
match $K_{12,2}^{\star}$ and $K_{12,3}^{\star}$, respectively. Thus $K = \Ks$, a contradiction. 
We conclude that $\Ks$ is the smallest solution in $\p \Gamma_+$. This proves that for the free endpoint case
when $Q \ge 0$ that $V_{\cL}(x_0) = x_0^{\top}P^- x_0$. Also, Theorem \ref{thm:mainresult} (iv) gives 
the optimal control $u(t) = -B^{\top}P^- x(t)$ since $P^- = \Ks$.
\end{proof}

Next we consider attainability in the free endpoint case. Since Assumption~\ref{assum:wellposed} holds, we can 
apply Theorem~\ref{thm:mainresult}(iii). In the free endpoint problem, $\cL_1 = \cC$, so by Theorem~\ref{thm:mainresult}(iii),
the problem is attainable if and only if $\Ker (\Delta_1) \subset \Ker (K_1^-)$. By Proposition 6.4 of \cite{TRE89}, 
the latter condition always holds. Thus, we recover the well-known fact that for the free endpoint case in the positive
semidefinite case, the problem is always attainable. 

Now we discuss the fixed endpoint problem. The main results are summarized in Theorem~10.18 in \cite{TRE02}. 
In particular, when $\cL = 0$, $V_{\cL}(x_0) = x_0^{\top} P^+ x_0$, where $P^+ \geq 0$ is the largest positive semidefinite 
solution to the ARE, and the optimal control is $u^{\star}(t) = - B^{\top} P^+ x(t)$. 
We would like to verify that our Theorem~\ref{thm:mainresult} recovers these results. We must 
show that when $Q \ge 0$, then $\Ks = P^+$. For the fixed endpoint problem, $\cL_1 = 0$, so 
$\cN_1(\cL_1) = 0$. The desired result is then immediately obtained from Theorem~\ref{thm:N1L1Zero}.

Now we consider attainability in the fixed endpoint case. 
The well-known necessary and sufficient conditions for attainability in the positive semidefinite case, stated in 
Theorem 10.18(iii) of \cite{TRE02}, is that every eigenvalue of $A$ on the imaginary axis is $(Q,A)$ observable. 
We must show that this statement is equivalent to our attainability result in Theorem~\ref{thm:mainresult}(iii), 
which for the fixed-endpoint case requires that $\Ker(\Delta_1) \subset 0 \cap \Ker(K_1^-)$, or equivalently, 
$\Delta_1 > 0$. This connection is resolved by the following result, whose proof is found in the Appendix. 

\begin{theorem}
\label{thm:delta1pos}
Suppose $(A,B)$ is stabilizable and $Q \geq 0$. Then every eigenvalue of $A$ on the imaginary axis is $(Q,A)$ 
observable if and only if $\Delta_1 > 0$.
\end{theorem}

The final verification of our result in the fixed endpoint case is to show that the closed-loop system, 
$\dot{x}(t) = (A - BB^{\top}K^+)x(t) = A(K^+)x(t)$, is asymptotically stable, thereby recovering Theorem 10.18(v) 
in \cite{TRE02}. Note that $A(K) = A-BB^{\top}K = 
\begin{bmatrix}A_1(K_1) & * \\ 0 & A_2 \end{bmatrix}$ so that $\sigma(A(K)) = \sigma(A_1(K_1)) \uplus \sigma(A_2)$. 
By Theorem~5 in  \cite{WIL71}, we have that $\Delta_1 > 0$ if and only if $\sigma(A_1(K_1^+)) \subset \CC^-$. 
Since $\sigma(A_2) \subset \CC^-$ by stabilizability and $\Delta_1 > 0$ by attainability, we have 
$\sigma(A(K^+)) \subset \CC^-$, as desired.

\section{Conclusion}
In this paper we address a problem in the area of linear quadratic optimal control which has been open for the last 20 years. Specifically, we consider the regular, infinite-horizon, stability-modulo-a-subspace, indefinite LQ problem when the dynamics are stabilizable. Previous works have also addressed this problem, but under the restrictive assumption that the dynamics are controllable. The generalization from controllable to stabilizable dynamics is significant in that there is a lack of structure in the solutions of the algebraic Riccati equation in the stabilizable case. Consequently the connection between the ARE solution set and the LQ problem under consideration has remained elusive. We resolved this gap by combining a suitable sufficient condition for a finite optimal cost with a specific  decomposition to unambiguously deduce the correct form of the optimal cost and control.
The determination of necessary and sufficient conditions for a finite value function in the regular, infinite-horizon, stability-modulo-a-subspace, indefinite LQ problem is still open. As future work, we are also interested in applying our result to reachability problems, namely by employing an indefinite cost functional on a stabilizable linear system to characterize the convergence of  trajectories to a nontrivial subspace over the infinite time horizon.

\bibliographystyle{siamplain}

\begin{thebibliography}{99}

\bibitem{ALBERT69}
A. Albert.
Conditions for Positive and Nonnegative Definiteness in Terms of Pseudoinverses.
{\em SIAM J. Applied Mathematics}.
vol. 17, no. 2, pp. 434-440, 1969.

\bibitem{MOO89}
B. D. O. Anderson and J. B. Moore.
{\em Optimal Control: Linear Quadratic Methods}. 
Prentice-Hall International, Inc., 1989.

\bibitem{BROCKETT70}
R. W. Brockett.
{\em Finite Dimensional Linear Systems}.
1970.

\bibitem{ENG08}
J. Engwerda.
The Regular Convex Cooperative Linear Quadratic Control Problem.
{\em Automatica}.
vol. 44, no. 9, pp. 2453-2457, 2008.

\bibitem{GANTMACHER}
F. R. Gantmacher.
{\em The Theory of Matrices}. 
vol. 1. Chelsea Publishing, 1959.

\bibitem{GEE88}
T. Geerts.
A Necessary and Sufficient Condition for Solvability of the Linear-Quadratic Control Problem without Stability.
{\em Systems and Control Letters}.
vol. 11, no. 1, pp. 47-51, 1988.

\bibitem{GEE89}
T. Geerts.
{\em Structure of Linear-Quadratic Control}.
Ph. D. Thesis, Eindhoven University of Technology, Eindhoven 1989.

\bibitem{GEE91}
T. Geerts.
A Priori Results in Linear-Quadratic Optimal Control Theory.
{\em Kybernetika}.
vol. 27, no. 5, pp. 446-457, 1991.

\bibitem{KUC72}
V. Kucera.
A Contribution to Matrix Quadratic Equations.
{\em IEEE Trans. Automatic Control}.
vol. 17, no. 3, pp. 344-347, 1972.

\bibitem{GOH86}
I. Gohberg, P. Lancaster, and L. Rodman.
On Hermitian Solutions of the Symmetric Algebraic Riccati Equation.
{\em SIAM J. Control and Optimization}.
vol. 24, no. 6, pp. 1323-1334, 1986.

\bibitem{LAN95}
P. Lancaster and L. Rodman.
{\em Algebraic Riccati Equations}.
Clarendon Press, Oxford, 1995.

\bibitem{MOL77}
B. P. Molinari. 
The Time-Invariant Linear-Quadratic Optimal Control Problem.
{\em Automatica}.
vol. 13, no. 4, pp. 347-357, 1977.

\bibitem{PACH09}
M. Pachter.
Revisit of Linear-Quadratic Optimal Control.
{\em J. Optimization Theory and Applications}.
vol. 140, no. 2, pp. 301-314, 2009.

%\bibitem{MOO00}
%M. A. Rami, X. Y. Zhou, and J. B. Moore.
%Well-posedness and attainability of indefinite stochastic linear quadratic control in infinite time horizon.
%{\em Systems and Control Letters}.
%vol. 41, no. 2, pp. 123-133, 2000.

%\bibitem{KWO86}
%T. J. Richardson and R. H. Kwong.
%On positive definite solutions to the algebraic Riccati equation.
%{\em Systems and Control Letters}.
%vol. 7, no. 4, pp. 99-104, 1986.

\bibitem{SCH91}
C. Scherer.
The Solution Set of the Algebraic Riccati Equation and the Algebraic Riccati Inequality.
{\em Linear Algebra and its Applications}.
vol. 153, pp. 99-122, 1991.

\bibitem{SCHM83}
J. M. Schumacher.
The Role of the Dissipation Matrix in Singular Optimal Control.
{\em Systems and Control Letters}.
vol. 2, no. 5, pp. 262-266, 1983.

%\bibitem{SHA83}
%M. A. Shayman.
%Geometry of the Algebraic Riccati Equation-Part 1.
%{\em Siam J. Control and Optimization}.
%vol. 21, no. 3, pp. 375-393, 1983.

\bibitem{TRE89b}
J.M. Soethoudt and H. L. Trentelman.
The Regular Indefinite Linear-Quadratic Problem with Linear Endpoint Constraints.
{\em Systems and Control Letters}.
vol. 12, no. 1, pp. 23-31, 1989.

\bibitem{TRE89}
H. L. Trentelman.
The Regular Free-Endpoint Linear Quadratic Problem with Indefinite Cost.
{\em SIAM J. Control and Optimization}.
vol. 27, no. 1, pp. 27-42, 1989.

\bibitem{TRE02}
H. Trentelman, A. A. Stoorvogel, and M. Hautus.
{\em Control Theory for Linear Systems}.
Springer, London, 2001.

\bibitem{WIL71}
J. C. Willems.
Least Squares Stationary Optimal Control and the Algebraic Riccati Equation.
{\em IEEE Trans. Automatic Control}.
vol. 16, no. 6, pp. 621-634, 1971.

\bibitem{WIL86}
J. C. Willems, A. Kitap\c{c}i, and L. M. Silverman.
Singular Optimal Control: A Geometric Approach.
{\em SIAM J. Control and Optimization}.
vol. 24, no. 3, pp. 323-337, 1986.

\bibitem{WONHAM68}
W. M. Wonham.
On a Matrix Riccati Equation of Stochastic Control.
{\em SIAM J. Control}.
vol. 6, no. 4, pp. 681-697, 1968.

\bibitem{WONHAM85}
W. M. Wonham.
{\em Linear Multivariable Control: A Geometric Approach}.
Springer-Verlag, New York, 3rd Edition, 1985. 

%\bibitem{RAMI01}
%M. A. Rami and J. B. Moore,
%``Partial Stabilizability and Hidden Convexity of an Indefinite LQ Problem,"
%August 20, 2001 (not published).

%\bibitem{WIM95}
%H. K. Wimmer,
%``Lattice Properties of Sets of Semidefinite Solutions of Continuous-time Algebraic Riccati Equations,"
%{\em Automatica},
%vol. 31, no. 2, pp. 173-182, 1995.

%\bibitem{KAL60}
%R. E. Kalman,
%``Contributions to the Theory of Optimal Control,"
%{\em Mex. Math. Soc. Bull.}
%vol. 5, pp. 102-119, 1960.

%\bibitem{BOY04}
%S. Boyd and L. Vandenberghe,
%``Convex Optimization,"
%Cambridge University Press, first edition, 2004.

\end{thebibliography}

\begin{appendix}
%\section{Appendix}
\appendixnotitle

\begin{theorem}[Theorem 1, \cite{ALBERT69}]
\label{thm:schur}
Given a real symmetric matrix $P = \begin{bmatrix} P_1 & P_{12} \\ P_{12}^{\top} & P_2 \end{bmatrix}$, 
the following conditions are equivalent:
\begin{enumerate}
\item $P \geq 0$.
\item $P_1 \geq 0, \; \; (I - P_1P_1^{\dagger})P_{12} = 0,        \; \; P_2 - P_{12}^{\top}P_1^{\dagger}P_{12} \geq 0$. 
\item $P_2 \geq 0, \; \; (I - P_2P_2^{\dagger})P_{12}^{\top} = 0, \; \; P_1 - P_{12}P_2^{\dagger}P_{12}^{\top} \geq 0$. 
\end{enumerate}
\end{theorem}

\begin{lemma} 
\label{lem:3by3blockschur}
Let $M$ be a symmetric positive semidefinite matrix with the block form
\begin{equation*}
M = \begin{bmatrix}M_1 & M_{12} \\ M_{12}^{\top} & M_2 \end{bmatrix} 
  = \begin{bmatrix} 0 & 0 & M_{12,1} \\ 0 & M_{1,22} & M_{12,2} \\ M_{12,1}^{\top} & M_{12,2}^{\top} & M_2 \end{bmatrix}.
\end{equation*}
\noindent Then $M_{12,1} = 0$.
\end{lemma}

\begin{proof}
Since $M \geq 0$, Theorem \ref{thm:schur} in particular implies that $(I - M_1 M_1^{\dagger})M_{12} = 0$. Using the properties of the pseudo-inverse, it can be shown that $M_1^{\dagger} = \begin{bmatrix} 0 & 0 \\ 0 & M_{1,22}^{\dagger} \end{bmatrix}$. Then the result follows from
\begin{equation*}
\begin{bmatrix} 0 \\ 0 \end{bmatrix} = \left(I - \begin{bmatrix} 0 & 0 \\ 0 & M_{1,22} \end{bmatrix} \begin{bmatrix} 0 & 0 \\ 0 & M_{1,22}^{\dagger}\end{bmatrix}\right)\begin{bmatrix}M_{12,1} \\ M_{12,2} \end{bmatrix} = \begin{bmatrix} I & 0 \\ 0 & I - M_{1,22}M_{1,22}^{\dagger} \end{bmatrix}\begin{bmatrix}M_{12,1} \\ M_{12,2} \end{bmatrix}.
\end{equation*}
\end{proof}

The following results are required for the proof of Theorem~\ref{thm:delta1pos}. 
Consider the Kalman controllable decomposition \eqref{eq:sysctrdecomp}.
First, we define the {\em Hamiltonian matrix} on the controllable subspace: 
\begin{equation}
H_1 := \begin{bmatrix} A_1 & -B_1 B_1^{\top} \\ -Q_1 & -A_1^{\top} \end{bmatrix}.
\end{equation}

\begin{lemma} 
\label{prop:observResults}
Let $A = \begin{bmatrix} A_1 & A_{12} \\ 0 & A_2 \end{bmatrix}$ with $\sigma(A_2) \subset \CC^-$ and $C = \begin{bmatrix}C_1 & C_2 \end{bmatrix}$. Then
\begin{enumerate}
\item[(i)] All of the eigenvalues of $A$ on the imaginary axis are $(C,A)$ observable if and only if all of the eigenvalues of $A_1$ on the imaginary axis are $(C_1,A_1)$ observable.
\item[(ii)] An eigenvalue of $A$ is $(C,A)$ observable if and only if it is $(C^{\top}C,A)$ observable.
\end{enumerate}
\end{lemma}

\begin{theorem}[\cite{MOL77}] 
\label{thm:HamDelta}
Let $(A_1,B_1)$ be controllable. The following conditions are equivalent.
\begin{enumerate}
\item[(i)] 
The maximal and minimal solutions of the ARE (and ARI), $K_1^+$ and $K_1^-$ respectively, exist and $\Delta_1 > 0$.
\item[(ii)] 
The Hamiltonian matrix has no pure imaginary eigenvalues, i.e., if $\lambda 
\in \sigma(H_1)$ then $\Re(\lambda) \neq 0$. 
\end{enumerate}
\end{theorem}

\begin{lemma}[Lemma 8, \cite{KUC72}] 
\label{lem:eigHam}
Suppose that $Q_1 = C_1^{\top}C_1$. Then there is an eigenvalue $\lambda$ of $H_1$ such that $\Re(\lambda) = 0$ if and only if there is an uncontrollable eigenvalue $\lambda $ of $(A_1,B_1)$ and/or unobservable eigenvalue $\lambda$ of $(C_1,A_1)$ such that $\Re(\lambda) = 0$.
\end{lemma}

\begin{proof}[Proof of Theorem~\ref{thm:delta1pos}]
First we check that $\Delta_1 = K_1^+ - K_1^-$ is well-defined. Since $(A,B)$ is stabilizable and $Q \geq 0$ 
implies $0 \in \Gamma_-$, we can apply Lemma~\ref{lem:wellposedgivesnegsdefK1} to get that 
$K_1^- \subset \p \Gamma_{1 - } \subset \p \Gamma_1 \neq \emptyset$, and so Theorem~\ref{thm:extremal}(i) 
establishes the existence of $K_1^+$ as well.

Since $Q \geq 0$, write $Q = C^{\top}C$. In the Kalman controllability decomposition, this means 
$C = \begin{bmatrix} C_1 & C_2 \end{bmatrix}$, so that $Q_1 = C_1^{\top}C_1$. 

First we can use Theorem \ref{thm:HamDelta} to establish $\Delta_1 > 0$ is equivalent to $H_1$ having no pure 
imaginary eigenvalues. The contrapositive of the Lemma \ref{lem:eigHam} says that $H_1$ has no pure imaginary 
eigenvalues if and only if there is no uncontrollable eigenvalue of $(A_1,B_1)$ {\em and} no unobservable eigenvalue of $(C_1,A_1)$ on the imaginary axis. Since in our scenario $(A_1,B_1)$ is controllable, this statement is equivalent to $H_1$ has no pure imaginary eigenvalues if and only if there are no unobservable eigenvalues of $(C_1,A_1)$ on the imaginary axis. Of course, there are no unobservable eigenvalues of $(C_1,A_1)$ on the imaginary axis if and only if all the eigenvalues of $A_1$ on the imaginary axis are $(C_1,A_1)$ observable. 
Applying Lemma~\ref{prop:observResults}(i), we get that all the eigenvalues of $A_1$ on the imaginary axis are $(C_1,A_1)$ observable if and only if all the eigenvalues of $A$ on the imaginary axis are $(C,A)$ observable. Then applying Lemma~\ref{prop:observResults}(ii), all the eigenvalues of $A$ on the imaginary axis are $(C,A)$ observable if and only if all the eigenvalues of $A$ on the imaginary axis are $(C^{\top}C,A)$ observable.
Since $Q = C^{\top}C$, we have proven every eigenvalue of $A$ on the imaginary axis is $(Q,A)$ observable if and only if $\Delta_1 > 0$ by a long chain of equivalent statements.
\end{proof}

\end{appendix}
\end{document}